\documentclass[12pt]{article} 

\usepackage[T1]{fontenc}
\usepackage[outline]{contour}
\usepackage{xcolor}

\usepackage[a4paper]{geometry}

\usepackage{amsmath,amsfonts,amsthm,amssymb,mathabx}
\usepackage{mathrsfs}
\usepackage{upgreek}
\usepackage{mathtools}
\usepackage[titletoc]{appendix}
\usepackage{bm}

\usepackage{graphicx}

\usepackage{url}

\usepackage{caption}

\usepackage{hyperref}
\hypersetup{colorlinks=false,pdfborder=0 0 0}

\newtheorem{theorem}{Theorem}[section]
\newtheorem{lemma}[theorem]{Lemma}
\newtheorem{proposition}[theorem]{Proposition}
\newtheorem{corollary}[theorem]{Corollary}

\newtheorem{remark}[theorem]{Remark}
\newtheorem{definition}[theorem]{Definition}

\newcommand*{\defeq}{\mathrel{\vcenter{\baselineskip0.5ex \lineskiplimit0pt
                     \hbox{\scriptsize.}\hbox{\scriptsize.}}}%
                     =}
                     
\usepackage[none]{hyphenat}
\usepackage{tikz}
\usepackage{tikz-cd}
\usetikzlibrary{matrix}
\usetikzlibrary{arrows}
\usepackage{wrapfig}

\newcommand{\alg}[1]{\mathbb{#1}}

\newcommand{\KH}[1]{K^{#1}_H}

\newcommand{\JH}[1]{J^{#1}_H}
\newcommand{\JHP}[1]{J^{#1}_{H+}}

\newcommand{\rep}[1]{\phi^{#1}}
\newcommand{\repH}[1]{\phi^{#1}_H}
\newcommand{\kap}[1]{\kappa^{#1}}
\newcommand{\kapH}[1]{\kappa^{#1}_H}
\newcommand{\prim}[1]{{\rep{#1}}'}
\newcommand{\primH}[1]{{\repH{#1}}'}

\newcommand{\xiG}[1]{\xi^{#1}}
\newcommand{\xiH}[1]{\xi^{#1}_H}
\newcommand{\weil}[1]{\omega^{#1}}
\newcommand{\weilH}[1]{\omega^{#1}_H}
\newcommand{\spec}[1]{\nu^{#1}}
\newcommand{\specH}[1]{\nu^{#1}_H}

\newcommand{\N}[1]{N^{#1}}

\newcommand{\Z}[1]{Z^{#1}}

\newcommand{\W}[1]{W^{#1}}
\newcommand{\WH}[1]{W^{#1}_H}
\newcommand{\Heis}[1]{\mathcal{H}^{#1}}
\newcommand{\HeisH}[1]{\mathcal{H}^{#1}_H}

\contourlength{0.15pt}
\contournumber{10}
\newcommand{\Bold}[1]{\contour{black}{#1}}

\DeclareMathOperator{\Ext}{Ext}
\DeclareMathOperator{\Fr}{Fr}
\DeclareMathOperator{\Sp}{Sp}

\DeclareMathOperator{\ad}{ad}

\DeclareMathOperator{\Res}{Res}
\DeclareMathOperator{\Ind}{c-Ind}

\DeclareMathOperator{\Hom}{Hom}
\DeclareMathOperator{\inn}{inn}
\DeclareMathOperator{\red}{\mathrm{red}}

\DeclareMathOperator{\val}{val}
\DeclareMathOperator{\Lie}{Lie}

\DeclareMathOperator{\Gal}{Gal}

\DeclareMathOperator{\res}{\mathfrak{f}} 
\DeclareMathOperator{\resun}{\overline{\mathfrak{f}}} 
\DeclareMathOperator{\iso}{\iota}  

\DeclareMathOperator{\p}{\mathfrak{p}} 
\DeclareMathOperator{\g}{\mathfrak{g}}
\DeclareMathOperator{\h}{\mathfrak{h}}

\newcommand{\z}{\mathfrak{z}^i} 

\newcommand{\zstar}{\mathfrak{z}^{i,*}}
\newcommand{\zstarri}{\mathfrak{z}^{i,*}_{-r_i}} 
\newcommand{\zder}{\mathfrak{z}^i_H} 
\newcommand{\zderstar}{\mathfrak{z}_H^{i,*}} 
\newcommand{\zderstarri}{(\mathfrak{z}_H^{i,*})_{-r_i}} 
\newcommand{\der}{\mathrm{der}}
\newcommand{\un}{\mathrm{un}}
\newcommand{\Gder}{G_{\der}} 

\newcommand{\torus}{\mathfrak{t}} 
\newcommand{\tder}{\mathfrak{t}_H} 

\newcommand{\Bigslant}[3]{#1\backslash #2 / #3}

\newcommand{\quo}[2]{#1/#2}

\usepackage[nottoc,notlot,notlof]{tocbibind}

\usepackage{fancyhdr}

\makeatletter
\def\blfootnote{\gdef\@thefnmark{}\@footnotetext}
\makeatother

\begin{document}

\title{Restricting Supercuspidal Representations via a Restriction of Data}
\author{Ad\`ele Bourgeois}
\date{}

\maketitle

\begin{abstract}

Let $F$ be a non-archimedean local field of residual characteristic $p$. Let $\alg{G}$ be a connected reductive group defined over $F$ which splits over a tamely ramified extension and set $G=\alg{G}(F)$. We assume that $p$ does not divide the order of the Weyl group of $\mathbb{G}$. Given a closed connected $F$-subgroup $\alg{H}$ that contains the derived subgroup of $\alg{G}$, we study the restriction to $H$ of an irreducible supercuspidal representation $\pi=\pi_G(\Psi)$ of $G$, where $\Psi$ is a $G$-datum as per the J.K.Yu Construction. We provide a full description of $\pi|_H$ into irreducible components, with multiplicity, via a restriction of data which constructs $H$-data from $\Psi$. Analogously, we define a restriction of Kim-Yu types to study the restriction of irreducible representations of $G$ which are not supercuspidal.

\end{abstract}

\blfootnote{This research was supported by the National Sciences and Engineering Research Council of Canada through the Alexander Graham Bell Canada Graduate Scholarship, as well as the Ontario Graduate Scholarship.}
 
\tableofcontents

\section{Introduction}

One of the fundamental questions of representation theory is to study how a completely reducible representation decomposes into irreducible subrepresentations (or components), and to what multiplicity these components appear. In particular, given a connected reductive group $\alg{G}$ defined over a non-archimedean local field $F$ of residual characteristic $p$, we are interested in this paper to study the restriction of an irreducible (smooth) representation $\pi$ of $G = \alg{G}(F)$ to a particular type of subgroup $H$. This subgroup $H$ is such that it is the $F$-points of a closed connected $F$-subgroup $\alg{H}$ of $\alg{G}$ which contains the derived subgroup, $\alg{G}_{\der}$, of $\alg{G}$. The restriction $\pi|_H$ is known to be completely reducible and to decompose into finitely many components \cite[Lemma 2.1]{Tadic:1992}.

Similar types of restrictions have been studied in the past. Indeed, one of the motivations for this paper comes from \cite[Conjecture 2.6]{AP:2006}, in which Adler and Prasad conjectured that multiplicity one holds for the restriction of irreducible admissible representations of $G$ to $H$ in the case where $\alg{H}$ is quasi-split and $\alg{G}$ and $\alg{H}$ are not necessarily connected. This conjecture was successfully proven for various pairs ($\alg{G},\alg{H}$), such as $(\mathrm{GL}(n),\mathrm{SL}(n))$ \cite[Theorem 1.3]{AP:2006}, $(\mathrm{GO}(n),\mathrm{O}(n))$ and $(\mathrm{GSp}(2n),\mathrm{Sp}(2n))$ \cite[Theorem 1.4]{AP:2006}, $(\mathrm{GU}(n),\mathrm{U}(n))$ \cite[Theorem 9]{AP:2018} and $(\mathrm{GSO}(n),\mathrm{SO}(n))$ \cite{Gee}. Adler and Prasad actually discovered their conjecture to be false by finding a counterexample of a depth-zero supercuspidal representation of $\mathrm{GU}(n)$ whose restriction to $\mathrm{SU}(n)$ restricts with multiplicity two \cite[Theorem 10]{AP:2018}. They discuss in \cite[Remark 11]{AP:2018} that the failure of the multiplicity one result is likely due to the fact that the components of the restriction are not \emph{regular} in the sense of \cite{Kaletha:2019}. 

In this paper, we pursue this question of restriction and focus particularly on the case where $\pi$ is an irreducible supercuspidal representation. This is because Jacquet's subrepresentation theorem tells us that such representations are the building blocks for the representation theory of $p$-adic groups such as $G$, and also because we know precisely how the irreducible supercuspidal representations are constructed. Notably, an exhaustive construction for the ones of depth zero was described by Moy and Prasad in \cite{MP:1996}, and a construction for irreducible supercuspidal representations of arbitrary depth was later described by Yu \cite{Yu:2001}, referred to as the J.K.~Yu Construction, which builds on the work done by Adler \cite{Adler:1998}. To apply such constructions, we assume that $\alg{G}$ splits over a tamely ramified extension $E$ of $F$. We also assume that $p$ does not divide the order of $W$, the Weyl group of $\alg{G}$, so that all irreducible supercuspidal representations of $G$ are obtained from the J.K.~Yu Construction \cite[Theorem 8.1]{Fintzen:2018}, a refinement of Kim's earlier exhaustion result \cite{Kim:2007}. Note that for all irreducible root systems, one can explicitly calculate $|W|$ as illustrated in \cite[Table 1]{Fintzen:2018}, and see that the imposed condition on $p$ is not extremely restrictive. When $p\nmid |W|$, we have that $p$ is neither a bad nor torsion prime \cite[Lemma 2.2]{Fintzen:2018}. It also implies that $p$ does not divide the order of the fundamental group of $\alg{G}_{\der}$ \cite[Section 3.7.4]{Kaletha:2019}, and therefore hypothesis $C(\alg{G})$ from \cite{HM:2008} (and the weaker hypothesis $C(\alg{G})_w$ from \cite{Murnaghan:2011}) holds \cite[Lemma 3.5.2]{Kaletha:2019}. 

The J.K.~Yu Construction starts from a $G$-datum $\Psi = (\vec{\alg{G}},y,\rho,\vec{\phi})$ and produces an irreducible supercuspidal representation $\pi_G(\Psi)$ of $G$. The components of $\pi_G(\Psi)|_{H}$ are irreducible supercuspidal representations of $H$, which means they are constructed from $H$-data via the exhaustion. As such, our strategy is to describe the components of $\pi_G(\Psi)|_H$ via the data. More specifically, we look at how the $H$-data associated to the components of $\pi_G(\Psi)|_H$ are related to the initial $G$-datum $\Psi$. In \cite{Nevins:2015}, Nevins described the relationship between the $H$-data and $\Psi$ in the very specific case where $H=\Gder$ and $\Psi$ is a \emph{toral datum of length one}. The results that we present in this paper significantly generalize those of \cite{Nevins:2015}. 

The first main result that we present (Theorem \ref{th:GderData}) is showing that there is a very natural restriction process that we can apply to $\Psi$ in order to construct a family of $H$-data, $\Res(\Psi)$. We also verify that this restriction of data is compatible with Hakim and Murnaghan's notion of equivalence of data from \cite{HM:2008}.
We then show that the steps of the J.K.~Yu Construction essentially commute with restriction (Theorem~\ref{th:phiprime}, Proposition~\ref{prop:kappai}), allowing us to explicitly compute $\pi_G(\Psi)|_H$ in Theorem~\ref{th:restrictSupercuspidal}. Indeed, given a set of representatives $C$ of $\Bigslant{H}{G}{K}$, where $K$ is the inducing subgroup from the J.K.~Yu Construction, we show that the supercuspidal representations constructed from $\{{^t\Res(\Psi):t\in C}\}$ exhaust all the components of $\pi_G(\Psi)|_H$. Conversely, one can also extend an $H$-datum $\Psi_H$ to a family of $G$-data (Theorem~\ref{th:backwards}) and use Frobenius Reciprocity to obtain a description of the components of $\Ind_H^G\pi_H(\Psi_H)$ (Theorem~\ref{th:induceSupercuspidal}). 


Because we are able to obtain such a precise description of $\pi_G(\Psi)|_H$, we can then study the multiplicity of each component. It turns out that the multiplicity in $\pi_G(\Psi)|_H$ is solely determined by $\rho$, the depth-zero piece of the datum $\Psi$ (Theorem~\ref{th:multiplicity}). Indeed, $\rho$ is an irreducible representation of $G^0_{[y]}$, where $G^0$ is the first twisted Levi subgroup of the sequence $\vec{\alg{G}}$, which induces to a depth-zero supercuspidal representation of $G^0$. The multiplicity in $\rho|_{H^0_{[y]}}$, where $H^0 = G^0\cap H$, is what determines the multiplicity in $\pi_G(\Psi)|_H$. In particular, $\pi_G(\Psi)|_H$ is multiplicity free if and only if $\rho|_{H^0_{[y]}}$ is multiplicity free.  Furthermore, each component of $\pi_G(\Psi)|_H$ appears with the same multiplicity, and this multiplicity is calculated explicitly in Proposition~\ref{prop:multiplicityrhoH} and is expressed in terms of the dimensions of $\rho$ and one of the components of $\rho|_{H^0_{[y]}}$.

Even though we are able to write a formula for the multiplicity, it can still be difficult to compute in practice. We consider the special case where the representation $\rho$ induces to a depth-zero supercuspidal representation which is \emph{regular} in the sense of \cite{Kaletha:2019}. 
In this case, we have a very specific description for $\rho$ which allows us to study the restriction $\rho|_{H^0_{[y]}}$ in great detail and obtain the conditions for a multiplicity free restriction in Corollary~\ref{cor:rhoMultFree} which confirms Adler and Prasad's observations that some notion of regularity is required. We obtain as a corollary (Corollary~\ref{cor:multiplicityRegular}) a multiplicity free result for the restriction of a regular depth-zero supercuspidal representation, which is also shown by Adler and Mishra in a recent pre-print \cite[Theorem 5.3]{AM:2019} using a different approach. 

Finally, for the irreducible representation $\pi$ of $G$ which is non-supercuspidal, we approach the problem of restricting $\pi$ to $H$ from the perspective of types. With our underlying assumption on $p$, such a representation contains a Kim-Yu type \cite[Theorem 7.12]{Fintzen:2018}, which is a refinement of Kim and Yu's exhaustion result \cite[Theorem 9.1]{KimYu}. Given that the construction of a Kim-Yu type is completely analogous to Yu's construction of supercuspidal representations \cite{KimYu}, we show in Theorems~\ref{th:typeData} and \ref{th:typeRes} that our methods for restricting the $G$-datum $\Psi$ and $\pi_G(\Psi)$ can be applied in this context, resulting in a restriction of types. In turn, this restriction of types provides us with information on $\pi|_H$.

The condition that $p \nmid |W|$ not only provides an exhaustion of the J.K.~Yu Construction and Kim-Yu types, it also allows for some simplifications when proving the main results of this paper. In particular, in the definition of genericity for characters (Definition~\ref{def:genericity}), one can omit what Yu refers to as condition GE2) \cite[Lemma 8.1]{Yu:2001}. Furthermore, the imposed condition on $p$ implies that we were able to use results that require hypothesis $C(\alg{G})$ from \cite{HM:2008} (or the weaker hypothesis $C(\alg{G})_w$ from \cite{Murnaghan:2011}). When there is no underlying assumption on $p$, one needs to take into account this condition GE2) and can follow Kaletha's strategy from \cite[Section 3.5]{Kaletha:2019} of using $z$-extensions to bypass hypotheses $C(\alg{G})$ and $C(\alg{G})_w$ whenever they are necessary. As a consequence, Theorems~\ref{th:GderData}, \ref{th:restrictSupercuspidal}, \ref{th:multiplicity}, \ref{th:typeData} and \ref{th:typeRes} all hold without any underlying assumption on $p$. This means that whenever $\pi$ is a supercuspidal representation of $G$ obtained from the J.K.~Yu Construction, all components of $\pi|_H$ also arise from the J.K.~Yu Construction. One could then wonder whether or not it is possible for some components of $\pi|_H$ to be obtained via the J.K.~Yu Construction when $\pi$ is not, and ask in what sense do the irreducible supercuspidal representations obtained from the J.K.~Yu Construction form a closed class of representations. We do not answer the previous question in this paper, but rather leave it as an open problem.

This paper is divided as follows. In Section~\ref{sec:strucTheory}, we start by establishing the relevant notation with regards to the structure theory of $G$, and show how that structure theory relates to that of $H$ in a very natural way. In Section~\ref{sec:reviewSC}, we review the definition of a $G$-datum as well as the steps of the J.K.~Yu Construction. Section~\ref{sec:restrictingData} is fully dedicated to our restriction process of data, whereas Section~\ref{sec:restrictSC} computes the restriction of $\pi_G(\Psi)$. Results related to the multiplicity in $\pi_G(\Psi)|_H$ are presented in Section~\ref{sec:multiplicity}, and Section~\ref{sec:regularDZ} looks at the example when $\rho$ induces to a regular depth-zero supercuspidal representation. Results in this last section are complemented by the Appendix, which presents properties of Deligne-Lusztig cuspidal representations. Finally, we discuss the restriction of non-supercuspidal representations via a restriction of types in Section~\ref{sec:types}. This paper contains the main results of my doctoral thesis \cite{thesis} and have been tailored to the more expert audience. 

\section{Notation and Structure Theory}\label{sec:strucTheory}

We begin this section by recalling all the relevant structure theory and notation for $G$ that is required for the construction of supercuspidal representations. We then establish important relations between the structure theory of $G$ and that of $H$, which will be crucial for the various types of restrictions we will be considering in this paper.

\subsection{Structure Theory of $G$}

Given the non-archimedean local field $F$, we denote by $\val(\cdot)$ its valuation which maps into $\mathbb{Z}$,  $\mathcal{O}_F$ its ring of integers, $\p_F$ the unique maximal ideal of $\mathcal{O}_F$ and $\res$ its residue field of prime characteristic $p$. Let $F^\un$ be a maximal unramified extension of $F$. It is well know that the residue field of $F^\un$ is an algebraic closure of $\res$, so we denote it by $\resun$.

Unless otherwise specified, $\alg{G}$ always denotes a connected reductive group which is defined over $F$, and we set $G = \alg{G}(F)$ and $\alg{G}_{\der} = [\alg{G},\alg{G}]$. We assume that $\alg{G}$ splits over a tamely ramified extension $E$ of $F$. We also set \Bold{$\g$} to be the Lie algebra of $\alg{G}$ and $\g = \text{\Bold{$\g$}}(F)$. Given a maximal $F$-split torus $\alg{T}$ of $\alg{G}$, we denote the corresponding affine apartment and reduced affine apartment of $G$ by $\mathcal{A}(\alg{G},\alg{T},F)$ and $\mathcal{A}^{\mathrm{red}}(\alg{G},\alg{T},F)$, respectively. We recall that $\mathcal{A}^{\mathrm{red}}(\alg{G},\alg{T},F) = \mathcal{A}(\alg{G}_{\der},\alg{T}\cap\alg{G}_{\der},F)$ and $\mathcal{A}(\alg{G},\alg{T},F) = \mathcal{A}^{\mathrm{red}}(\alg{G},\alg{T},F)\times (X_*(Z(\alg{G}),F)\otimes_\mathbb{Z} \mathbb{R})$, where $Z(\alg{G})$ denotes the center of $\alg{G}$. The Bruhat-Tits building and reduced building of $G$ are denoted by $\mathcal{B}(\alg{G},F)$ and $\mathcal{B}^{\mathrm{red}}(\alg{G},F)$, respectively. Similarly to the apartment, we have that $\mathcal{B}^{\mathrm{red}}(\alg{G},F) = \mathcal{B}(\alg{G}_{\der},F)$ and $\mathcal{B}(\alg{G},F) = \mathcal{B}^{\mathrm{red}}(\alg{G},F)\times (X_*(Z(\alg{G}),F)\otimes_\mathbb{Z}\mathbb{R})$.

For each $x\in\mathcal{B}(\alg{G},F), r > 0$, $G_{x,r}$ denotes the Moy-Prasad filtration subgroup of the parahoric subgroup $G_{x,0}$. We also set $G_{x,r^+} = \underset{t >r}{\bigcup}G_{x,t}$. We use colons to abbreviate quotients, that is $G_{x,r:t} = \quo{G_{x,r}}{G_{x,t}}$ for $t > r$. We have analogous filtrations of $\mathcal{O}_F$-submodules at the level of the Lie algebra.

\begin{lemma}\label{lem:[x]}
Given $x\in\mathcal{B}(\alg{G},F)$, let $[x]\defeq \{x+z : z\in X_*(Z(\alg{G}),F)\otimes_\mathbb{Z}\mathbb{R}\}$. Then for all $y\in [x], r\geq 0$, we have $G_{x,r} = G_{y,r}$.
\end{lemma}

\begin{proof}
This follows from the fact that the Moy-Prasad filtration subgroups at $x$ are determined by how roots evaluate on $x$, and that roots evaluate trivially on $Z(\alg{G})$ \cite[Corollary 26.2B]{Humphreys:1975:LAG}.
\end{proof}

Note that the set $[x]$ from the previous lemma is viewed as a point of the reduced building $\mathcal{B}^{\mathrm{red}}(\alg{G},F)$ when viewing $\mathcal{B}^{\mathrm{red}}(\alg{G},F)$ as a quotient of $\mathcal{B}(\alg{G},F)$.


For all $r>0$, the quotient $G_{x,r:r^+}$ is an abelian group and is isomorphic to its Lie algebra analog $\g_{x,r:r^+}$. Adler constructs an isomorphism $e: \g_{x,r:r^+} \rightarrow G_{x,r:r^+}$ in \cite[Section 1.5]{Adler:1998} with a theory of mock exponential maps. This isomorphism is used in the J.K.~Yu Construction.

The quotient $G_{x,0:0^+}$ is also very important, as it results in the $\res$-points of a reductive group. We shall denote this group by $\mathcal{G}_x$ and refer to it as the \emph{reductive quotient of $\alg{G}$ at $x$}. In other words, $\mathcal{G}_x$ is the reductive group such that $\mathcal{G}_x(\res) = G_{x,0:0^+}$.

\subsection{Structure Theory of $H$ in Relation to $G$}\label{sec:HandG}

From this point forward, $\alg{H}$ denotes a closed connected $F$-subgroup of $\alg{G}$ which contains $\alg{G}_{\der}$. Not only is this choice partly motivated by the earlier conjecture \cite[Conjecture 2.6]{AP:2006}, it implies that $\alg{H}$ is a normal subgroup of $\alg{G}$ and that $[\alg{H},\alg{H}] = \alg{G}_{\der}$. It also means that all the important subgroups of $\alg{G}$ and $\alg{H}$ necessary for studying representations are related in the most natural way, via intersection. The same follows for the important subgroups of $G$ and $H$.

\begin{theorem}\label{th:intersectTorus}
Let $\alg{T}$ be a maximal torus of $\alg{G}$. Then $\alg{T}\cap\alg{H}$ is a maximal torus of $\alg{H}$. Furthermore, every maximal torus of $\alg{H}$ is of the form $\alg{T}\cap\alg{H}$ for some unique maximal torus $\alg{T}$ of $\alg{G}$.
\end{theorem}

\begin{proof}
Because $\alg{T}$ is a torus, $\alg{T} \simeq \alg{G}_m^n \simeq \alg{D}(n,\alg{G}_m)$ for some $n > 0$, where $\alg{D}(n,\alg{G}_m)$ denotes the group of invertible $n\times n$ diagonal matrices with entries in $\alg{G}_m$. The properties of $\alg{H}$ then imply that $\alg{T}\cap\alg{H}$ is isomorphic to a closed and connected subgroup of $\alg{D}(n,\alg{G}_m)$, therefore making it a torus \cite[Theorem 16.2]{Humphreys:1975:LAG}. Now, given a maximal torus $\alg{S}$ of $\alg{H}$, the conjugacy of maximal tori of $\alg{G}$ implies that $\alg{S} \subset {^g\alg{T}}$ for some $g\in\alg{G}$. By maximality of $\alg{S}$, and by what precedes, it then follows that $\alg{S} = {^g\alg{T} \cap\alg{H}}$. Normality of $\alg{H}$ then guarantees that $^g\alg{T}\cap\alg{H}$ is a maximal torus of $\alg{H}$ for all $g\in\alg{G}$. Furthermore, one can use the normality of $\alg{H}$ and the central isogeny $\alg{G} = Z(\alg{G})^\circ\alg{G}_{\der}$ to show that $^g\alg{T}$ is the unique maximal torus of $\alg{G}$ that satisfies $\alg{S} = {^g\alg{T}\cap\alg{H}}$.
\end{proof}

Let $\alg{T}$ be a maximal torus of $\alg{G}$ which is $E$-split and set $\alg{T}_\alg{H} = \alg{T}\cap\alg{H}$. Because $\alg{G}$ is reductive, we have a corresponding root system $\Phi = \Phi(\alg{G},\alg{T})$ and a presentation in terms of generators $\alg{G} = \langle \alg{T}, \alg{U}_\alpha : \alpha\in\Phi \rangle$, where $\alg{U}_\alpha, \alpha\in\Phi$, are the associated root subgroups \cite[Theorem 26.3]{Humphreys:1975:LAG}. As $\alg{G}$ is split over $E$, this presentation is preserved at the level of $E$-points. Because $\alg{T}$ normalizes the root subgroups, one can show that $\alg{G}_{\der} = \langle \alg{U}_\alpha: \alpha\in\Phi \rangle$. As a consequence, we have two ways of writing $\alg{G}$ into a product, $\alg{G} = Z(\alg{G})^\circ\alg{G}_{\der}$ and $\alg{G} = \alg{T}\alg{G}_{\der} = \alg{T}\alg{H}$.

One can verify that $\alpha|_{\alg{T}_\alg{H}}\in \Phi(\alg{H},\alg{T}_\alg{H})$ for all $\alpha\in\Phi(\alg{G},\alg{T})$ and that $\Phi(\alg{H},\alg{T}_\alg{H})$ and $\Phi(\alg{G},\alg{T})$ can be identified as root systems. Therefore, the root subgroups $\alg{U}_\alpha$ of $\alg{G}$ with respect to $\alg{T}$ are also root subgroups of $\alg{H}$ with respect to $\alg{T}_\alg{H}$, so that $\alg{H} = \langle \alg{T}_\alg{H}, \alg{U}_\alpha: \alpha\in\Phi \rangle$. It is then easy to see that $\alg{G}$ and $\alg{H}$ have the same Weyl groups, and we establish relationships between the Levi and parabolic subgroups of $\alg{G}$ with those of $\alg{H}$ as per the following theorem.

\begin{theorem}\label{th:intersectLevi}
Let $\alg{M}$ be a Levi subgroup of $\alg{G}$ and $\alg{P}$ be a parabolic subgroup of $\alg{G}$. Then $\alg{M}\cap\alg{H}$ is a Levi subgroup of $\alg{H}$ and $\alg{P}\cap\alg{H}$ is a parabolic subgroup of $\alg{H}$. Furthermore, every Levi and parabolic subgroup of $\alg{H}$ arises uniquely this way.
\end{theorem}

\begin{proof}
Without loss of generality, one can assume that $\alg{M}$ and $\alg{P}$ are standard Levi and parabolic subgroups of $\alg{G}$, meaning that they are generated by a maximal torus $\alg{T}$ and some of the root subgroups $\alg{U}_\alpha, \alpha\in\Phi = \Phi(\alg{G},\alg{T})$. Because $\alg{T}$ normalizes the root subgroups, it is easy to see that $\alg{M}\cap\alg{H}$ and $\alg{P}\cap\alg{H}$ are generated by the maximal torus $\alg{T}_\alg{H}= \alg{T}\cap\alg{H}$ and root subgroups $\alg{U}_\alpha$, meaning that they are standard Levi and parabolic subgroups of $\alg{H}$. That all Levi and parabolic subgroups of $\alg{H}$ arise uniquely this way follows from the uniqueness of $\alg{T}$ given $\alg{T}_\alg{H}$ given by Theorem~\ref{th:intersectTorus}.
\end{proof}

As a consequence of this last theorem, we obtain that the restriction to $H$ of a supercuspidal representation of $G$ is again supercuspidal, as the unipotent subgroups of their parabolic subgroups coincide.

The argument from Theorem~\ref{th:intersectLevi} can be used for other types of subgroups which are generated by a maximal torus and root subgroups. For instance, given a semisimple element $s\in \alg{H}$, the previous argument can be applied on the description provided for $C_{\alg{G}}(s)^\circ$ in \cite[Theorem 3.5.3]{Carter:1993} to deduce $C_{\alg{G}}(s)^\circ\cap\alg{H} = C_{\alg{H}}(s)^\circ$. We may also apply this type of argument to show that $H_{x,r} = G_{x,r}\cap H$ for all $r\geq 0$.

\begin{remark}\label{rem:WLOGy}
Because $[\alg{H},\alg{H}] = \alg{G}_{\der}$, given a maximal $F$-split torus $\alg{T}$ of $\alg{G}$, we have that $\mathcal{A}^{\mathrm{red}}(\alg{G},\alg{T},F) = \mathcal{A}^{\mathrm{red}}(\alg{H},\alg{T}_\alg{H},F)$, where $\alg{T}_\alg{H} = \alg{T}\cap\alg{H}$, and ${\mathcal{B}^{\mathrm{red}}(\alg{G},F) = \mathcal{B}^{\mathrm{red}}(\alg{H},F)}$. We then have a natural embedding
$\mathcal{A}(\alg{H},\alg{T}_\alg{H},F)\hookrightarrow \mathcal{A}(\alg{G},\alg{T},F)$
via the inclusion ${X_*(Z(\alg{H}),F)\subset X_*(Z(\alg{G}),F)}$. So, given a point $x\in\mathcal{A}(\alg{G},\alg{T},F)$, one can write things like $H_{x,r}$ and $H_{x,r^+}$, $r\geq 0$, without loss of generality by Lemma~\ref{lem:[x]}. 
\end{remark}

As a consequence of having $H_{x,r} = G_{x,r}\cap H$ for all $r\geq 0$, it follows from the second isomorphism theorem that $H_{x,r:r^+} \simeq \quo{H_{x,r}G_{x,r^+}}{G_{x,r^+}}$, where this isomorphism maps $hH_{x,r^+}$ to $hG_{x,r^+}$ for all $h\in H_{x,r}$. As a result, we obtain the following two lemmas.

\begin{lemma}\label{lem:isores}
Let $e_H: \h_{x,r:r^+}\rightarrow H_{x,r:r^+}$ be Adler's isomorphism from \cite[Section 1.5]{Adler:1998}. Then we may view $e_H$ as the restriction to $\h_{x,r}$ of $e$ and we write $e|_{\h_{x,r}} = e_H$.
\end{lemma}

\begin{proof}
First note that the map $e$ can be viewed as a map $\g_{x,r} \rightarrow G_{x,r:r^+}$ with kernel $\g_{x,r^+}$.

One simply needs to follow the steps of Adler's construction of $e$ in \cite[Section 1.5]{Adler:1998}, which Hakim summarized into a diagram in \cite[Section 3.4]{Hakim:2018}, and take restrictions to the corresponding subgroups of $\mathfrak{h}$ and $H$. This will map $\h_{x,r}$ into $\quo{H_{x,r}G_{x,r^+}}{G_{x,r^+}}$.

Letting $\iso$ be the isomorphism from $H_{x,r:r^+}$ to $\quo{H_{x,r}G_{x,r^+}}{G_{x,r^+}}$, we obtain $e|_{\h_{x,r}} = \iso\circ e_H$, which we write symbolically as $e|_{\h_{x,r}} = e_H$.
\end{proof}

\begin{lemma}\label{lem:HxGx}
Let $\mathcal{H}_x$ denote the reductive quotient of $\alg{H}$ at $x$. Then we may view $\mathcal{H}_x$ as a subgroup of $\mathcal{G}_x$ that contains $[\mathcal{G}_x,\mathcal{G}_x]$.
\end{lemma}

\begin{proof}
We identify the reductive groups with their $\resun$-points. By what precedes above, we have that $\mathcal{H}_x(\resun) \simeq \quo{\alg{H}(F^\un)_{x,0}\alg{G}(F^\un)_{x,0^+}}{\alg{G}(F^\un)_{x,0^+}} \subset \mathcal{G}_x(\resun)$. 

Furthermore,  $$[\mathcal{G}_x(\resun),\mathcal{G}_x(\resun)] = \quo{[\alg{G}(F^\un)_{x,0},\alg{G}(F^\un)_{x,0}]\alg{G}(F^\un)_{x,0^+}}{\alg{G}(F^\un)_{x,0^+}}.$$ The subgroup $[\alg{G}(F^\un)_{x,0},\alg{G}(F^\un)_{x,0}]$ is contained in both $\alg{G}(F^\un)_{x,0}$  and  ${\alg{G}_{\der}(F^\un) \subset \alg{H}(F^\un)}$. Therefore, $[\alg{G}(F^\un)_{x,0},\alg{G}(F^\un)_{x,0}]  \subset \alg{H}(F^\un)_{x,0}$, and ${[\mathcal{G}_x(\resun),\mathcal{G}_x(\resun)]\subset \quo{\alg{H}(F^\un)_{x,0}\alg{G}(F^\un)_{x,0^+}}{\alg{G}(F^\un)_{x,0^+}} \simeq \mathcal{H}_x(\resun)}$. 
\end{proof}

\section{Constructing Supercuspidal Representations}\label{sec:reviewSC}

In this section, we provide a summary of the J.K.~Yu Construction, starting from the definition of a $G$-datum $\Psi$, and describing the steps that need to be followed in order to construct the irreducible supercuspidal representation $\pi_G(\Psi)$. We also recall Hakim and Murnaghan's notion of $G$-equivalence for $\Psi$, which determines the equivalence class of $\pi_G(\Psi)$.

\subsection{The Datum}\label{sec:datum}

The datum for constructing supercuspidal representations is composed of five elements, each satisfying interrelated conditions \cite{Yu:2001,HM:2008}. We begin by stating
these conditions as axioms and will expand on some of the more technical axioms
after the statement of the definition.

\begin{definition}\label{def:datum}
A sequence $\Psi = (\vec{\alg{G}},y,\vec{r},\rho,\vec{\phi})$ is a (generic cuspidal) $G$-datum if and only if:

\begin{enumerate}
\item[D1)] $\vec{\alg{G}}$ is a tamely ramified twisted Levi sequence $\vec{\alg{G}}=(\alg{G}^0,\alg{G}^1,\dots ,\alg{G}^d)$ in $\alg{G}$, where $\alg{G}^d = \alg{G}$ and such that $\quo{Z(\alg{G}^0)}{Z(\alg{G})}$ is $F$-anisotropic;

\item[D2)] $y$ is a point in $\mathcal{B}(\alg{G},F)\cap \mathcal{A}(\alg{G},\alg{T},E)$, where $\alg{T}$ is a maximal torus of $\alg{G}^0$ (and maximal in all $\alg{G}^i$), and $E$ is a Galois tamely ramified splitting field of $\alg{T}$ (hence of $\vec{\alg{G}}$);

\item[D3)] $\vec{r} = (r_0,r_1,\dots , r_d)$ is a sequence of real numbers satisfying ${0<r_0<r_1<\cdots <r_{d-1}\leq r_d}$ if $d>0$, $0\leq r_0$ if $d=0$;

\item[D4)] $\rho$ is an irreducible representation of $K^0=G^0_{[y]}$ such that $\rho |_{G^0_{y,0^+}}$ is 1-isotypic and $\Ind\limits_{K^0}^{G^0}\rho$ is irreducible supercuspidal;

\item[D5)] $\vec{\phi} = (\rep{0},\rep{1},\dots ,\rep{d})$ is a sequence of quasicharacters, where $\rep{i}$ is a $G^{i+1}$-generic character of $G^i$ (relative to $y$) of depth $r_i$ for $0\leq i \leq d-1$. If $r_{d-1}<r_d$, we assume $\rep{d}$ is of depth $r_d$, otherwise $\rep{d}=1$.
\end{enumerate}
\end{definition}

Axioms D2 and D4 are related to Moy and Prasad's construction of depth-zero supercuspidal representations \cite[Theorem 6.8]{MP:1996}. The point $y$ is not arbitrary, and is such that $[y]$ is a vertex of $\mathcal{B}^{\red}(G^0,F)$ so that $G^0_{y,0}$ is a maximal parahoric subgroup of $G^0$. Furthermore, $K^0 = N_{G^0}(G^0_{y,0})$ \cite[Lemma 3.3]{Yu:2001}. Because $\rho$ induces to a depth-zero supercuspidal representation, this also means that $\rho|_{G^0_{y,0}}$ contains the pullback of a cuspidal representation of $G^0_{y,0:0^+}$. Having a depth-zero supercuspidal representation embedded in the datum means that the J.K.~Yu Construction builds on Moy and Prasad's construction, and contains it as a special case.

Axioms D3 and D5 are clearly intertwined. In fact, because $\vec{r}$ is implicit in $\vec{\phi}$, one can omit it from the datum and simply use a 4-tuple $\Psi = (\vec{\alg{G}},y,\rho,\vec{\phi})$. The notion of genericity is quite technical, and in order to define it, we introduce the notation from \cite[Section 3.1]{HM:2008}. Let \Bold{$\mathfrak{t}$} $= \Lie(\alg{T})$ and $\mathfrak{t} = $ \Bold{$\mathfrak{t}$}$(F)$. For all $0\leq i\leq d-1$ let \Bold{$\mathfrak{z}$}$^i$ denote the center of  \Bold{$\g$}$^i= \Lie(\alg{G}^i)$, and \Bold{$\mathfrak{z}$}$^{i,*}$ its dual. We set $\z =$ \Bold{$\mathfrak{z}$}$^i(F)$ and ${\zstar=\text{\Bold{$\mathfrak{z}$}}^{i,*}(F)}$. We also have $\z_{r_i} = \z\cap \torus_{r_i}$ and $\z_{r_i^+} = \z\cap \torus_{r_i^+}$ and define  
$$\zstarri= \{X^* \in \zstar: X^*(Y) \in \p_F \text{ for all } Y\in \z_{r_i^+}\}.$$ We also fix a character $\psi$ of $F$ which is nontrivial on $\mathcal{O}_F$ and trivial on $\p_F$. Given $a\in \Phi(\alg{G},\alg{T})$ we let $H_a = d\check{a}(1)$, where $\check{a}$ is the corresponding coroot. Given our underlying assumption on $p$, \cite[Lemma 8.1]{Yu:2001} implies that we can simplify the definition of genericity from \cite[Definition 3.9]{HM:2008} to the following.

\begin{definition}\label{def:genericity}
Set $0\leq i\leq d-1$. A quasicharacter $\chi$ of $G^i$ is $G^{i+1}$-generic of depth $r_i$ (relative to $y$) if and only if $\chi|_{G^i_{y,r_i^+}} = 1$ and there exists an element $X^*\in \zstarri$, satisfying $\val(X^*(H_a)) = -r_i$ for all $a\in\Phi(\alg{G}^{i+1},\alg{T})\setminus\Phi(\alg{G}^i,\alg{T})$, such that $X^*$ realizes the restriction of $\chi$ to $G^i_{y,r_i}$. That is,  $\chi(e(Y+\g^i_{y,r_i^+})) = \psi(X^*(Y))$ for all $Y\in \g^i_{y,r_i}$, where $e: \g^i_{y,r_i:r_i^+} \rightarrow G^i_{y,r_i:r_i^+}$ is Adler's isomorphism from \cite[Section 1.5]{Adler:1998}.
\end{definition}

\begin{remark}
In \cite[Definition 3.9]{HM:2008}, it explicitly mentions $\chi|_{G^i_{y,r_i}}\neq 1$ as a required condition for genericity. However, this is actually a consequence of having $\chi|_{G^i_{y,r_i}}$ realized by an element $X^*\in \zstarri$ as per Definition~\ref{def:genericity}.
\end{remark}

\begin{remark}
Under our assumption on $p$, a character $\chi$ of $G^i$ is $G^{i+1}$-generic relative to $y$ if and only if it is $G^{i+1}$-generic relative to $x$ for all $x\in\mathcal{B}(\alg{G}^i,F)$ \cite[Lemma 4.7]{Murnaghan:2011}.
\end{remark}

One interpretation we can give to the notion of genericity comes from \cite[Remark 6.4]{Murnaghan:2011} and is the following. If $\chi^i$ is a quasicharacter of $G^i$ which is $G^{i+1}$-generic of depth $r_i$, this means that $\chi^i|_{G^i_{y,r_i}}$ is not the restriction to $G^i_{y,r_i}$ of a quasicharacter of a twisted Levi subgroup $\dot{G}$ of $G^{i+1}$ such that $G^i\subset \dot{G}$. So in a sense, genericity means that we cannot obtain the restriction to $G^i_{y,r_i}$ from a bigger twisted Levi subgroup.

\subsection{Constructing $\pi_G(\Psi)$}\label{sec:construction}

From the $G$-datum $\Psi$, Yu constructs an open compact-mod-center subgroup $K^i$ of $G^i$ for all $0\leq i\leq d$. After setting $K^0 = G^0_{[y]}$, one can define $K^{i+1} \defeq K^0G^1_{y,s_0}\cdots G^{i+1}_{y,s_i}$ for all $0\leq i\leq d-1$, where $s_i = r_i/2$. Using $(\rho,\vec{\phi})$, Yu constructs an irreducible representation, $\kappa_G(\Psi)$, of $K^d$ such that $\pi_G(\Psi)\defeq \Ind_{K^d}^G\kappa_G(\Psi)$ is irreducible supercuspidal of depth $r_d$. We summarize the construction of $\kappa_G(\Psi)$ into three steps.  

\textbf{Step 1 (extension):} We view $\rep{i}$ as a character of $K^i$ by restriction. As a first step, Yu extends the character $\rep{i}$ to a (possibly higher-dimensional) representation $\prim{i}$ of $K^{i+1}$ for all $0\leq i\leq d-1$. This extension is nontrivial and requires the theory of Heisenberg groups and Weil representations in most cases. To extend the character $\rep{i}$ to a representation of $K^{i+1}$, Yu introduces subgroups of $\alg{G}(E)$, which are denoted $J^{i+1}(E)$ and $J^{i+1}_+(E)$, and are defined as follows:

\begin{align*}
J^{i+1}(E) = \langle \alg{T}(E)_{r_i}, \alg{U}_a(E)_{y,r_i}, \alg{U}_b(E)_{y,s_i} : a\in \Phi(\alg{G}^i,\alg{T}), b\in \Phi(\alg{G}^{i+1},\alg{T})\setminus \Phi(\alg{G}^i,\alg{T})  \rangle,\\
J^{i+1}_+(E) = \langle \alg{T}(E)_{r_i}, \alg{U}_a(E)_{y,r_i}, \alg{U}_b(E)_{y,s_i^+} : a\in \Phi(\alg{G}^i,\alg{T}), b\in \Phi(\alg{G}^{i+1},\alg{T})\setminus \Phi(\alg{G}^i,\alg{T})  \rangle,
\end{align*}
where the subgroup $\alg{U}_\alpha(E)_{y,r}$ denotes a filtration subgroup of $\alg{U}_\alpha(E)$.

We write $J^{i+1}$ and $J^{i+1}_+$ for $J^{i+1}(E)\cap G^{i+1}$ and $J^{i+1}_+(E)\cap G^{i+1}$, respectively. Yu also used the notation $(\alg{G}^i(E),\alg{G}^{i+1}(E))_{y,(r_i,s_i)}$ for $J^{i+1}(E)$, and $(G^i,G^{i+1})_{y,(r_i,s_i)}$ for $J^{i+1}$. Similarly, ${J^{i+1}_+ = (G^i,G^{i+1})_{y,(r_i,s_i^+)}}$. We have that $K^{i+1} = K^iJ^{i+1} = K^iG^{i+1}_{y,s_i}$ for all $0\leq i\leq d-1$ \cite[Section 3.1]{HM:2008}.

Given the previous group descriptions, in order to extend the character $\rep{i}$ of $K^i$ to $K^{i+1}$, we must find a way to extend it over $J^{i+1}$, or $G^{i+1}_{y,s_i}$. Yu explains how to extend $\rep{i}$ to a character $\widehat{\rep{i}}$ of $K^iG^{i+1}_{y,s_i^+}$ in \cite[Section 4]{Yu:2001}. Indeed, $\widehat{\rep{i}}$ is the (unique) character of $K^iG^{i+1}_{y,s_i^+}$ that agrees with $\rep{i}$ on $K^i$ and is trivial on $(G^i,G^{i+1})_{y,(r_i^+,s_i^+)}$ \cite[Section 3.1]{HM:2008}. This extension is simple because $s_i$ is chosen specifically so that it is the smallest number which makes $G^i_{y,s_i^+:r_i^+}$ and $G^{i+1}_{y,s_i^+:r_i^+}$ abelian, making it easy to extend $\rep{i}|_{G^i_{y,s_i^+}}$ over $G^{i+1}_{y,s_i^+}$.

This extends $\rep{i}$ to a character of $K^iG^{i+1}_{y,s_i^+} = K^iJ^{i+1}_+$, but we need to extend it further to $K^{i+1}=K^iJ^{i+1}$. There are two cases to consider.

\textit{Case 1:} When $J^{i+1}=J^{i+1}_+$, we have obtained the desired extension to $K^{i+1}$, which is defined by $\prim{i}(kj) = \rep{i}(k)\widehat{\rep{i}}(j)$ for all $k\in K^i, j\in J^{i+1}$.

\textit{Case 2:} When $J^{i+1}\neq J^{i+1}_+$, we cannot simply extend the quasicharacter $\rep{i}$ to be a quasicharacter of $K^iG^{i+1}_{y,s_i} = K^iJ^{i+1}$. This is because genericity (further) implies that $\rep{i}|_{G^i_{y,s_i}}$ cannot be extended to a character of $G^{i+1}_{y,s_i}$. We must then construct a higher dimensional representation. Yu does this using the theory of Heisenberg groups and Weil representations. The theory of Heisenberg groups and Weil representations is extensively discussed in \cite{Yu:2001,HM:2008,Nevins:2015}.

To summarize briefly, Yu shows that the quotient $\W{i} = \quo{J^{i+1}}{J^{i+1}_+}$ is a symplectic vector space. Let $\xiG{i} = \widehat{\rep{i}}|_{J^{i+1}_+}$ and set $\N{i} = \ker(\xiG{i})$. The quotient $\Heis{i} = \quo{J^{i+1}}{\N{i}}$ is a Heisenberg $p$-group with center $\Z{i} = \quo{J^{i+1}_+}{\N{i}}$ and such that $\quo{\Heis{i}}{\Z{i}} \simeq \W{i}$. Furthermore, $\Heis{i}\simeq \W{i}\boxtimes\Z{i}$, where $ \W{i}\boxtimes\Z{i}$ denotes couples $(w,z), w\in \W{i},z\in\Z{i}$ under a specific multiplication rule. Since $G^i_{y,r_i}\subset J^{i+1}_+$ and $\rep{i}$ is of depth $r_i$, this means that $\xiG{i}$ is a nontrivial central character. The version of the Stone-von Neumann Theorem as stated in \cite[Theorem 3]{MC:2012} says that, up to isomorphism, there is a unique irreducible representation $\eta_{\xiG{i}}$ of $\Heis{i}$ having $\xiG{i}$ as its central character.

Now, let $\Sp(\W{i})$ denote the symplectic group of $\W{i}$ 
and set $\Sp(\Heis{i})$ to be the set of automorphisms of $\Heis{i}$ that act trivially on $\Z{i}$. The group $\Sp(\W{i})$ can be viewed as a subgroup of $\Sp(\Heis{i})$ as $\Heis{i}\simeq\W{i}\boxtimes\Z{i}$. Given the representation $\eta_{\xiG{i}}$ of $\Heis{i}$ above, the Stone-von Neumann Theorem tells us that $^g\eta_{\xiG{i}} \simeq \eta_{\xiG{i}}$ for all $g\in \Sp(\W{i})$ since $^g\eta_{\xiG{i}}$ and $\eta_{\xiG{i}}$ both have $\xiG{i}$ as their central character. So, for all $g\in\Sp(\W{i})$, there exists an intertwining map $w(g)$ which intertwines $^g\eta_{\xiG{i}}$ and $\eta_{\xiG{i}}$. 

It turns out that the maps $w(g),g\in\Sp(\W{i})$, can be chosen compatibly to make a representation of $\Sp(\W{i})$, which is called the \emph{Weil representation}. This means that the representation $\eta_{\xiG{i}}$ can be naturally extended to a representation $\widehat{\eta_{\xiG{i}}}$ of ${\Sp(\W{i})\ltimes \Heis{i}}$, where $w=\widehat{\eta_{\xiG{i}}}|_{\Sp(\W{i})}$ and $\eta_{\xiG{i}}=\widehat{\eta_{\xiG{i}}} |_{\Heis{i}}$. 

The representation $\widehat{\eta_{\xiG{i}}}$ is called the \emph{Heisenberg-Weil lift} of $\eta_{\xiG{i}}$. When $p\neq 3$ and ${\dim\W{i} >2}$, the Heisenberg-Weil lift is unique. When $p=3$ and $\dim\W{i} = 2$, three lifts are possible, but Hakim and Murnaghan single out one of them as being the optimal choice \cite[Definition 2.17]{HM:2008}. 

Using the action by conjugation of $K^i$ on $J^{i+1}$, one can think of $K^i$ as a subgroup of $\Sp(\W{i})\subset \Sp(\Heis{i})$ up to homomorphism. Hence, composing this homomorphism with the Heisenberg-Weil lift obtained from the central character $\xiG{i}$, we obtain a new representation $\weil{i}$ of $K^i\ltimes \Heis{i}$. 

The representation $\prim{i}$ is then defined as $\prim{i}(kj) = \rep{i}(k)\weil{i}(k,j\N{i})$ for all $k\in K^i, j\in J^{i+1}$. Note that $\prim{i}$ is no longer a character in this case, but rather a representation whose dimension is related to the size of the quotient $\quo{J^{i+1}}{J^{i+1}_+}$.

\textbf{Step 2 (inflation):} For the second step, the representations $\prim{i}, 0\leq i\leq d-2$, are extended further to representations of $K^d$ by a process called inflation. The resulting representations are denoted by $\kap{i}, 0\leq i\leq d-2$. Given that $K^d = K^{i+1}J^{i+2}\dots J^d$, every element of $K^d$ is of the form $kj$ for some $k\in K^{i+1}, j\in J^{i+2}\cdots J^d$. For all $k\in K^{i+1}, j\in J^{i+2}\cdots J^d$, we define $\kap{i}(kj)\defeq\prim{i}(k)$ \cite[Section 3.4]{HM:2008}. We may also use the notation $\inf_{K^{i+1}}^{K^d}\prim{i}$ for $\kap{i}$. Note that the representation $\prim{i}$ can be inflated to any $K^j, j > i +1$.

For consistency of notation, we let $\kap{d-1} = {\rep{d-1}}'$ and $\kap{d} = \rep{d}$. The representation $\rho$ can also be inflated to $K^d$, and we denote its inflation by $\kappa^{-1}$, treating $\rho$ as the element of index $-1$ in the sequence $(\rho,\vec{\phi})$. 

\textbf{Step 3 (multiplication):}
After we have inflated all the extensions to $K^d$, we set ${\kappa_G(\Psi) = \kap{-1}\otimes \kap{0}\otimes \cdots\otimes \kap{d}}$. Figure~\ref{fig:JKYuConstruction} below summarizes the construction of $\kappa_G(\Psi)$. 

\begin{figure}[!htbp]
\center \begin{tikzpicture}
  \matrix (m) [matrix of math nodes,row sep=0.8em,column sep=0.9em,minimum width=0.1em]
  {
     {} & \text{$\left( \rep{0}, K^0\right)$}    & \text{$\cdots$} &\text{$\left(\rep{d-2},K^{d-2}\right)$}  & \text{$\left(\rep{d-1}, K^{d-1}\right)$}  & \text{$\left(\rep{d}, K^d\right)$}\\
        {} & {} & {} & {} & {}  &{} & {} \\
    \text{$\left(\rho, K^0\right)$} &\text{$\left(\prim{0}, K^1\right)$} & \text{$\cdots$} &\text{$\left(\prim{d-2},K^{d-1}\right)$}& \text{$\left(\prim{d-1},K^d\right)$}  &  \text{$\left(\prim{d}, K^d\right)$}\\
      {} & {} & {} & {}  & {}  & {}  & {} & {} \\
     \text{$\left(\kap{-1}, K^d\right)$} & \text{$\left(\kap{0},K^d\right)$}  & \text{$\cdots$} &\text{$\left(\kap{d-2},K^d\right)$}&  \text{$\left(\kap{d-1},K^d\right)$}  &  \text{$\left(\kap{d}, K^d\right)$}\\
    };
  \path[-stealth]
    (m-1-2) edge node [right] {extend} (m-3-2)
    (m-1-4) edge node [right] {extend} (m-3-4)
    (m-1-5) edge node [right] {extend} (m-3-5)
    (m-1-6) edge node [right] {=} (m-3-6)
    (m-3-1) edge node [right] {inflate} (m-5-1)
    (m-3-2) edge node [right] {inflate} (m-5-2)
    (m-3-4) edge node [right] {inflate} (m-5-4)
    (m-3-5) edge node [right] {=} (m-5-5)
    (m-3-6) edge node [right] {=} (m-5-6)
;
\end{tikzpicture}
$\kappa_G(\Psi) = \kap{-1} \otimes \kap{0} \otimes \kap{1} \otimes \kap{2} \cdots \otimes \kap{d-2} \otimes \kap{d-1} \otimes \kap{d}$
\caption[Summary of the J.K.~Yu Construction for $G$.]{Summary for the construction of $\kappa_G(\Psi)$ in the J.K.~Yu Construction.}
\label{fig:JKYuConstruction}
\end{figure}

\begin{remark}
From the $G$-datum $\Psi$, Yu's construction allows us to recursively obtain a supercuspidal representation $\pi^i$ of $G^i$ of depth $r_i$ for each $0\leq i\leq d$.

When $\Psi = (\alg{G},y,\rho,1)$, we see that $\kappa_G(\Psi) = \rho$ and therefore $\pi_G(\Psi) = \Ind_{G_{[y]}}^{G}\rho$ is a depth-zero supercuspidal representation \cite[Theorem 6.8]{MP:1996}.
\end{remark}

\subsection{Equivalence Classes}

It is possible for different $G$-data to produce the same supercuspidal representation, that is $\Psi\neq\dot{\Psi}$ and $\pi_G(\Psi)\simeq\pi_G(\dot{\Psi})$. Hakim and Murnaghan showed in \cite{HM:2008} that the equivalence class of the representation $\pi_G(\Psi)$ depends on what they call the $G$-equivalence class of $\Psi$. 

To define the notion of $G$-equivalence for data, Hakim and Murnaghan introduce three ways of altering a $G$-datum $\Psi = (\vec{\alg{G}},y,\rho,\vec{\phi})$ into a different $G$-datum $\dot{\Psi}$. The first alteration is $G$-conjugation, in which $\dot{\Psi} = {^g\Psi} = ({^g\vec{\alg{G}}},g\cdot y, {^g\rho}, {^g\vec{\phi}})$. The second alteration is called an elementary transformation \cite[Definition 5.2]{HM:2008}, in which $\dot{\Psi} = (\vec{\alg{G}},\dot{y},\dot{\rho},\vec{\phi})$, where $[\dot{y}] = [y]$ and $\dot{\rho}\simeq \rho$. The last alteration introduced is that of refactorization \cite[Definition 4.19]{HM:2008}, in which the sequences $(\rho,\vec{\phi})$ and $(\dot{\rho},\vec{\dot{\phi}})$ are essentially twists of each other by specific characters. 

Now, two $G$-data $\Psi$ and $\dot{\Psi}$ are said to be $G$-equivalent if $\dot{\Psi}$ can be obtained from $\Psi$ by a finite sequence of refactorizations, $G$-conjugations and elementary transformations \cite[Definition 6.1]{HM:2008}. Hakim and Murnaghan show how the $G$-equivalence class of $\Psi$ determines the equivalence class of $\pi_G(\Psi)$ in \cite[Theorems 6.6 and 6.7]{HM:2008}, which we summarize below.

\begin{theorem}[{\cite[Theorems 6.6 and 6.7]{HM:2008}}]\label{th:HM67}
Suppose $\Psi = (\vec{\alg{G}},y,\rho,\vec{\phi})$ and ${\dot{\Psi} = (\vec{\dot{\alg{G}}},\dot{y},\dot{\rho},\vec{\dot{\phi}})}$ are $G$-data. Set $\phi = \prod\limits_{i=0}^d\rep{i} |_{G^0}$, $\dot{\phi} = \prod\limits_{i=0}^d\dot{\rep{i}} |_{\dot{G}^0}$. Then the following are equivalent:
\begin{enumerate}
\item[1)] $\pi_G(\Psi)\simeq \pi_G(\dot{\Psi})$.
\item[2)] $\Psi$ and $\dot{\Psi}$ are $G$-equivalent.
\item[3)]There exists $g\in G$ such that $K^0 = {^{g}\dot{K^0}}$, $\vec{\alg{G}} = {^g\vec{\dot{\alg{G}}}}$ and $\rho\otimes\phi\simeq {^g(\dot{\rho}\otimes\dot{\phi})}$.
\end{enumerate}
\end{theorem}

\section{Restricting Data}\label{sec:restrictingData}

Given a closed connected $F$-subgroup $\alg{H}$ that contains $\alg{G}_{\der}$ and a $G$-datum $\Psi$, we can construct a family of $H$-data using a very natural restriction process, which is provided as the main theorem of this section. We also show that this restriction process is compatible with Hakim and Murnaghan's notion of equivalence of data.

\subsection{Main Theorem}

\begin{theorem}\label{th:GderData}
Let $\Psi = (\vec{\alg{G}},y,\vec{r},\rho,\vec{\phi})$ be a $G$-datum. 
Let $\alg{H}^i = \alg{G}^i\cap\alg{H}$ and ${\repH{i} = \rep{i} |_{H^i}}$ for all $0\leq i\leq d$. Let $\tilde{r}$ denote the depth of $\repH{d}$.
\begin{enumerate}
\item[1)] If $\tilde{r} > r_{d-1}$, set $\vec{\phi_H}=(\repH{0},\dots,\repH{d-1},\repH{d})$ and $\vec{\tilde{r}} = (r_0,\dots,r_{d-1},\tilde{r})$.
\item[2)] If $\tilde{r}\leq r_{d-1}$, set $\vec{\phi_H} = (\repH{0},\cdots,\repH{d-1}\repH{d},1)$ and $\vec{\tilde{r}} = (r_0,\dots, r_{d-1},r_{d-1})$.
\end{enumerate}
Then, for each irreducible subrepresentation $\rho_\ell$ of $\rho_H \defeq \rho |_{K^0_H}$ with $K^0_H = H^0_{[y]}$, ${\Psi_\ell = (\vec{\alg{H}},y,\vec{\tilde{r}},\rho_\ell,\vec{\phi_H})}$ is an $H$-datum, where $\vec{\alg{H}} = (\alg{H}^0,\dots,\alg{H}^d)$.
\end{theorem}

\begin{figure}[htbp!]
\center
\begin{tikzpicture}
  \matrix (m) [matrix of math nodes,row sep=2em,column sep=0.5em,minimum width=0.1em]
  {
    \text{$(\mathbb{G}^0,$} & \text{$\mathbb{G}^1,$} & \text{$\dots,$} & \text{$\mathbb{G}^d),$} &\text{\hspace{4mm}} &\text{$\rho,$} &\text{\hspace{4mm}}&\text{$(\rep{0},$} &\text{$\dots,$} &\text{$\rep{d-1},$} & \text{$\rep{d})$}\\
     \text{$(\mathbb{H}^0,$}  &\text{$\mathbb{H}^1,$} &\text{$\dots,$} &\text{$\mathbb{H}^d),$} &\text{\hspace{4mm}} &\text{$\rho_\ell$,} &\text{\hspace{4mm}} &\text{$(\repH{0}$,} &\text{$\dots,$} &\text{$\repH{d-1}$,} &\text{$\repH{d})$}\\};
  \path[-stealth]
  (m-1-1) edge node [right] {\small $\cap \alg{H}$} (m-2-1)
  (m-1-2) edge node [right] {\small $\cap \alg{H}$} (m-2-2)
  (m-1-4) edge node [right] {\small $\cap \alg{H}$} (m-2-4)
  (m-1-8) edge node [left] {\small $\Res^{G^0}_{H^0}$} (m-2-8)
  (m-1-10) edge node [left] {\small $\Res_{H^{d-1}}^{G^{d-1}}$} (m-2-10)
    (m-1-11) edge node [right] {\small $\Res^{G^{d}}_{H^{d}}$} (m-2-11)
  (m-2-11) edge[bend left=90] node [below] {$\tilde{r}\leq r_{d-1}$} (m-2-10);
      \draw[dash pattern=on5pt off3pt] (m-1-6) -- (m-2-6) ;
\end{tikzpicture}
\caption[Summary of the construction of $H$-data from a $G$-datum.]{Summary of the construction of $H$-data from a $G$-datum $\Psi$, where $\rho_\ell$ is any irreducible subrepresentation of $\rho_H$.}
\label{fig:constructDatum}
\end{figure}

Theorem~\ref{th:GderData} is a generalization of the results in \cite{Nevins:2015}, in which Nevins treated the case of a \emph{toral datum of length one} and $\alg{H}=\alg{G}_{\der}$. A toral datum of length one is a very simple case for which $\vec{\alg{G}} = (\alg{T},\alg{G})$, where $\alg{T}$ is a maximal torus of $\alg{G}$. In a toral datum of length one, it can be assumed without loss of generality that $\rho = 1$ via refactorization. 
Not only is the datum considered in \cite{Nevins:2015} simpler, but restricting to $\Gder$ also allows for some simplifications. Indeed, every character of $\Gder$ is of depth zero with our underlying assumption on $p$ \cite[Lemma 3.5.1]{Kaletha:2019}. As a consequence, we know that the depth of $\rep{1}|_{\Gder}$ will be smaller than $r$. We state Nevins' result as a corollary of our theorem.

\begin{corollary}[{\cite[Proposition 4.5]{Nevins:2015}}]\label{prop:NevinsDatum}
Let $\Psi = ((\alg{T},\alg{G}),y,r,1,(\rep{0},\rep{1}))$ be a toral datum of length one for $G$, where $r$ denotes the depth of the quasicharacter $\rep{0}$ of $T$ and $\rep{1}$ is a quasicharacter of $G$ which is either trivial, or of depth $\tilde{r} > r$. Let $\alg{T}_{\der} = \alg{T}\cap\alg{G}_{\der}$. Set $\rep{0}_{\der} = \rep{0} |_{T_{\der}}$ and $\rep{1}_{\der} = \rep{1} |_{\Gder}$. Define
$$\Psi_{\der} = ((\alg{T}_{\der},\alg{G}_{\der}),y,1,r,(\rep{0}_{\der}\rep{1}_{\der},1)).$$ Then $\Psi_{\der}$ is a toral datum of length one for $\Gder$.
\end{corollary}

Notice that in the case of the toral datum of length one, only one $\Gder$-datum was obtained by this restriction process. In our more general case, we see that multiple data may be produced, depending on how $\rho_H$ decomposes into irreducible subrepresentations.  

In order to prove Theorem~\ref{th:GderData}, we must show that ${\Psi_\ell = (\vec{\alg{H}},y,\vec{\tilde{r}},\rho_\ell,\vec{\phi_H})}$ satisfies the five axioms for an $H$-datum (Definition~\ref{def:datum}). Without loss of generality, one can assume that the point $y$ is in $\mathcal{B}^\mathrm{red}(\alg{G},F)\cap \mathcal{A}^\mathrm{red}(\alg{G},\alg{T},E)$ by Lemma~\ref{lem:[x]} and Remark~\ref{rem:WLOGy} so that Axiom D2 is automatic, and Axiom D3 is clearly satisfied. We split the proof of the remaining axioms into three propositions.

\begin{proposition}[Axiom D1]\label{prop:D1}
Let $\vec{\alg{H}} = (\alg{H}^0,\dots,\alg{H}^d)$, where $\alg{H}^i = \alg{G}^i\cap\alg{H}$ for all $0\leq i\leq d$. Then $\vec{\alg{H}}$ is a tamely ramified twisted Levi sequence in $\alg{H}$, where $\alg{H}^d=\alg{H}$ and such that $\quo{Z(\alg{H}^0)}{Z(\alg{H})}$ is $F$-anisotropic.
\end{proposition}

\begin{proof}
One easily sees that $\vec{\alg{H}}$ is a tamely ramified twisted Levi sequence of $\alg{H}$. This is because $\alg{H}^i(E)= \alg{G}^i(E)\cap\alg{H}(E)$ is a Levi subgroup of $\alg{H}$ by Theorem~\ref{th:intersectLevi}. 

Furthermore, we claim that $\quo{Z(\alg{G}^i)}{Z(\alg{G})} \simeq \quo{Z(\alg{H}^i)}{Z(\alg{H})}$ for all $0\leq i\leq d$ so that $\quo{Z(\alg{H}^0)}{Z(\alg{H})}$ is $F$-anisotropic. Indeed, from the central isogeny $\alg{G}=Z(\alg{G})^\circ\alg{G}_{\der}$, it follows that $\alg{G}=Z(\alg{G})^\circ\alg{H}$. Because $Z(\alg{G})\subset \alg{T}\subset \alg{G}^i$, we obtain that $\alg{G}^i = Z(\alg{G})^\circ\alg{H}^i$ for all $0\leq i\leq d$ when intersecting the last equality with $\alg{G}^i$. 

Now, with what precedes, we show that $Z(\alg{G}^i) = Z(\alg{G})^\circ Z(\alg{H}^i)$ for all $0\leq i\leq d$. Given $x\in Z(\alg{G}^i)$, we have that $x=zy$ for some $z\in Z(\alg{G})^\circ, y\in \alg{H}^i$. Then for all $z'\in Z(\alg{G})^\circ, y'\in \alg{H}^i$, we have $x(z'y') = (z'y')x$ which implies that $yy'=y'y$ as $z$ and $z'$ are central. Therefore, $y\in Z(\alg{H}^i)$ and $Z(\alg{G}^i) \subset Z(\alg{G})^\circ Z(\alg{H}^i)$. For the converse, it is clear that $Z(\alg{G})^\circ \subset Z(\alg{G}^i)$ as $Z(\alg{G})^\circ\subset \alg{G}^i$. Furthermore, $Z(\alg{H}^i) \subset Z(\alg{G}^i)$ as a consequence of Theorem~\ref{th:intersectLevi}. Indeed, this theorem provides us with descriptions of $\alg{G}^i$ and $\alg{H}^i$ in terms of generators that only differ by their tori $\alg{T}$ and $\alg{T}_\alg{H}$, and $\alg{T}$, which contains $Z(\alg{H}^i)$, is abelian. Thus, $Z(\alg{G}^i) = Z(\alg{G})^\circ Z(\alg{H}^i)$.

Using the previous equality, we have $\quo{Z(\alg{G}^i)}{Z(\alg{G})} = \quo{Z(\alg{G})^\circ Z(\alg{H}^i)}{Z(\alg{G})^\circ Z(\alg{H})}$ for all $0\leq i\leq d$.
One can then easily check that $\quo{Z(\alg{G})^\circ Z(\alg{H}^i)}{Z(\alg{G})^\circ Z(\alg{H})} \simeq \quo{Z(\alg{H}^i)}{Z(\alg{H})} $ by considering that map that sends $zZ(\alg{G})^\circ Z(\alg{H})$ to $zZ(\alg{H})$ for all $z\in Z(\alg{H}^i)$. In particular, this isomorphism holds for $i=0$ and therefore $\quo{Z(\alg{H}^0)}{Z(\alg{H})}$ is $F$-anisotropic.
\end{proof}

Before proceeding to the proposition that proves Axiom D4, let us first make a few observations in the form of lemmas.

\begin{lemma}
For all $0\leq i\leq d$, $\alg{H}^i$ is a closed connected $F$-subgroup of $\alg{G}^i$ that contains $[\alg{G}^i,\alg{G}^i]$ and therefore $\alg{H}^i$ is normal in $\alg{G}^i$ and $[\alg{H}^i,\alg{H}^i] = [\alg{G}^i,\alg{G}^i]$.
\end{lemma}

\begin{proof}
For all $0\leq i \leq d$, $\alg{H}^i$ is a Levi subgroup, which is in particular a reductive subgroup so that it is closed and connected. Furthermore, $[\alg{G}^i,\alg{G}^i]\subset \alg{G}^i\cap\alg{G}_{\der} \subset \alg{G}^i\cap\alg{H} = \alg{H}^i$.
\end{proof}

\begin{lemma}\label{lem:[y]}
The normalizer of $H^0_{y,0}$ in $H^0$ is equal to $H^0_{[y]}$.
\end{lemma}

\begin{proof}
Because $[\alg{H}^0,\alg{H}^0]=[\alg{G}^0,\alg{G}^0]$, $\mathcal{B}^{\mathrm{red}}(\alg{H}^0,F) =\mathcal{B}^{\mathrm{red}}(\alg{G}^0,F)$ so that $[y]$ is a vertex of $\mathcal{B}^\mathrm{red}(\alg{H}^0,F)$. It follows that $H^0_{y,0}$ is a maximal parahoric subgroup of $H^0$ and ${\KH{0} = H^0_{[y]} = N_{H^0}(H^0_{y,0})}$ \cite[Lemma 3.3]{Yu:2001}.
\end{proof}

\begin{lemma}\label{lem:rhofinitedim}
Let $\rho_H = \rho|_{H^0_{[y]}}$. We have a decomposition into irreducible subrepresentations $\rho_H = \underset{\ell\in L}{\oplus}{\rho_\ell}$ for some finite set $L$.
\end{lemma}

\begin{proof}
We have that $\rho_H$ is a finite-dimensional representation \cite[Remark 7.2]{Murnaghan:2011}.  To complete the proof, we must show that $\rho_H$ is completely reducible. As $H^0_{y,0}$ is compact, $\rho|_{H^0_{y,0}}$ is completely reducible. Since $\rho$ is irreducible, the center of $G^0$ acts by a character. In particular, the center $Z^0_H$ of $H^0$ acts by a character so that any $H^0_{y,0}$-invariant subspace is also $Z^0_HH^0_{y,0}$-invariant. Therefore, $\rho|_{Z^0_HH^0_{y,0}}$ is completely reducible. Because $\KH{0}$ is compact-mod-center in $H^0$, it easily follows that $Z^0_H\tilde{K}$ is of finite index in $\KH{0}$ for any open subgroup $\tilde{K}\subset \KH{0}$. In particular, $Z^0_HH^0_{y,0}$ is of finite index in $\KH{0}$, so $\rho_H = \rho|_{\KH{0}}$ is completely reducible by \cite[Lemma 2.7]{BH:2006}.
\end{proof}

\begin{proposition}[Axioms D4]\label{prop:D4}
For all $\ell\in L$, $\rho_\ell|_{H^0_{y,0^+}}$ is $1$-isotypic and $\pi^{-1}_\ell \defeq \Ind_{\KH{0}}^{H^0}\rho_\ell$ is an irreducible supercuspidal representation of depth zero.
\end{proposition}

\begin{proof}
By \cite[Theorem 6.8]{MP:1996}, it suffices to show that $\rho_\ell|_{H^0_{y,0^+}}$ contains the pullback of a cuspidal representation of $H^0_{y,0:0^+}$.

Since $\Ind_{K^0}^{G^0}\rho$ is an irreducible supercuspidal representation of depth zero, we know from \cite[Theorem 6.8]{MP:1996} that $\rho|_{G^0_{y,0}}$ contains the pullback of a cuspidal representation of $G^0_{y,0:0^+}$, say $\sigma$. By Frobenius Reciprocity, it follows that $\Hom_{K^0}(\Ind_{G^0_{y,0}}^{K^0}\sigma,\rho) \neq \{0\}$, or equivalently, we write $\rho \subset \Ind_{G^0_{y,0}}^{K^0}\sigma$ to mean that $\rho$ is a subrepresentation of $\Ind_{G^0_{y,0}}^{K^0}\sigma$. It follows that $\rho|_{G^0_{y,0}} \subset \Res_{G^0_{y,0}}^{K^0}\Ind_{G^0_{y,0}}^{K^0}\sigma$. But $G^0_{y,0}$ is a normal subgroup of $K^0$, so the Mackey Decomposition gives $\Res_{G^0_{y,0}}^{K^0}\Ind_{G^0_{y,0}}^{K^0}\sigma = \underset{c\in C}{\oplus}{^c\sigma}$, where $C$ is a set of coset representatives of $\quo{K^0}{G^0_{y,0}}$. We conclude that $\rho|_{G^0_{y,0}}\subset \underset{c\in C}{\oplus} {^c\sigma}$, which implies that $\rho_\ell|_{H^0_{y,0}} \subset \underset{c\in C}{\oplus} ^c(\sigma|_{H^0_{y,0}})$.

Next, we claim that $\sigma|_{H^0_{y,0}}$ decomposes as a sum of cuspidal representations of $H^0_{y,0:0^+}$. Indeed, because $[\mathcal{G}^0_y,\mathcal{G}^0_y] \subset \mathcal{H}^0_y$ (Lemma~\ref{lem:HxGx}), it follows from Theorem~\ref{th:intersectLevi} that $\sigma|_{\mathcal{H}^0_y(\res)}$ decomposes as a direct sum of cuspidal representations of $\mathcal{H}^0_y(\res)$, say $\underset{j\in J}{\oplus} \sigma_j$ for some finite set $J$.

Therefore, $\rho_\ell|_{H^0_{y,0}} \subset \underset{c\in C}{\oplus} ^c(\underset{j\in J}{\oplus}\sigma_j)$. Since $\rho_\ell|_{H^0_{y,0}}$ is completely reducible, being a finite-dimensional representation of a compact group, we conclude from the uniqueness (up to isomorphism) of the decomposition into irreducible representations that $\rho_\ell$ must contain the pullback to $H^0_{y,0}$ of one of the cuspidal representations $^c\sigma_j$.
\end{proof}

Finally, in order to prove Axiom D5, we must establish notation to describe the genericity condition. Analogously to the notation from Section~\ref{sec:datum}, let \Bold{$\torus$}$_{\alg{H}} = \Lie(\alg{T}_{\alg{H}})$, $\tder =$ \Bold{$\torus$}$_{\alg{H}}(F)$ and $T_H = \alg{T}_{\alg{H}}(F)$. For ${0\leq i\leq d-1}$, let \Bold{$\mathfrak{z}$}$^i_{\alg{H}}$ denote the center of \Bold{$\h$}$^i = \Lie(\alg{H}^i)$, and \Bold{$\mathfrak{z}$}$^{i,*}_{\alg{H}}$ its dual. We set $\zder =$\Bold{$\mathfrak{z}$}$^i_{\alg{H}}(F)$ and $\zderstar=$\Bold{$\mathfrak{z}$}$^{i,*}_{\alg{H}}(F)$. We also have $(\zder)_{r_i} = \zder\cap (\tder)_{r_i}$ and ${(\zder)_{r_i^+} = \zder\cap (\tder)_{r_i^+}}$ and define
$$\zderstarri = \{X^* \in \zderstar: X^*(Y)\in \p_F \text{ for all } Y\in (\zder)_{r_i^+}\}.$$

We note that given an element from $\zstar$, one can view it as an element of ${\g}^{i,*}$ by extending it trivially on ${\g^i}' $, where ${\g^i}' = [$\Bold{$\g^i$},\Bold{$\g^i$}$](F) = [$\Bold{$\h^i$},\Bold{$\h^i$}$](F)$. This follows from the fact that ${\g}^i = \z\oplus{\g^i}' $ \cite[Proposition 3.1]{Adler:2000}.

\begin{proposition}[Axiom D5]\label{prop:D5}
The sequence $\vec{\phi_H}$ of quasicharacters of $\vec{\alg{H}}$ is such that the $i$th quasicharacter of the sequence is $H^{i+1}$-generic of depth $r_i$ for all $0\leq i \leq d-1$.
\end{proposition}

\begin{proof}
We start by showing that the characters $\repH{i}$ are $H^{i+1}$-generic of depth $r_i$ for all ${0\leq i\leq d-1}$.

It is clear that $\repH{i}|_{H^i_{y,r_i^+}} = 1$ as $\rep{i}|_{G^i_{y,r_i^+}}=1$.

Now, let $X^*$ be an element of $\zstarri$ satisfying $\val(X^*(H_a)) = -r_i$ for all ${a\in\Phi(\alg{G}^{i+1},\alg{T})\setminus\Phi(\alg{G}^i,\alg{T})}$ that realizes $\rep{i}|_{G^i_{y,r_i}}$ as per Definition~\ref{def:genericity}. We claim that $X^*|_{\h^i}$ is an element of $\zderstarri$, satisfying $\val(X^*|_{\h^i}(H_a)) = -r_i$ for all $a\in\Phi(\alg{H}^{i+1},\alg{T}_{\alg{H}})\setminus \Phi(\alg{H}^i,\alg{T}_{\alg{H}})$, such that $X^*|_{\h^i}$ realizes $\repH{i}|_{H^i_{y,r_i}}$. Indeed, $X^*|_{\h^i}$ is clearly an element of $\zderstar$ viewed as an element of $\h^{i,*}$. Since $X^*\in \zstarri$, we have that $X^*(Y)\in\p_F$ for all $Y\in \z_{r_i^+}$. In particular, ${X^*(Y)\in\p_F}$ for all $Y\in(\zder)_{r_i^+}$. Hence, $X^*|_{\h^i} \in \zderstarri$. Recall from Section~\ref{sec:HandG} that we may identify the root systems $\Phi(\alg{G}^i,\alg{T})$ and $\Phi(\alg{H}^i,\alg{T}_\alg{H})$. Since $\val(X^*(H_a)) = -r_i$ for all $a\in \Phi(\alg{G}^{i+1},\alg{T})\setminus \Phi(\alg{G}^i,\alg{T})$, we conclude that $\val(X^*|_{\h^i}(H_a)) = -r_i$ for all $a\in\Phi(\alg{H}^{i+1},\alg{T}_{\alg{H}})\setminus \Phi(\alg{H}^i,\alg{T}_{\alg{H}})$.

Now, we show that $X^*|_{\h^i}$ realizes $\repH{i}$, that is for all $Y\in \h^i_{y,r_i}$, ${\repH{i}(e_H(Y+\h^i_{y,r_i^+})) = \psi(X^*|_{\h^i}(Y))}$, where $e_H$ is Adler's isomorphism between $\h^i_{y,r_i:r_i^+}$ and $H^i_{y,r_i:r_i^+}$ from \cite[Section 1.5]{Adler:1998}. Take $Y\in \h^i_{y,r_i}$. Since $\phi^i$ is realized by $X^*$, we have $\psi(X^*|_{\h^i}(Y))= \rep{i}(e(Y+\g^i_{y,r_i^+}))$ as per Definition~\ref{def:genericity}. By Lemma~\ref{lem:isores}, it follows that $e(Y+\g^i_{y,r_i^+}) = e_H(Y + \h^i_{y,r_i^+}),$ and therefore $\psi(X^*|_{\h^i}(Y)) = \repH{i}( e_H(Y+\h^i_{y,r_i^+}))$.
Hence, $\repH{i}$ is $H^{i+1}$-generic of depth $r_i$ for all $0\leq i\leq d-1$.

Finally, our definition of $\vec{\phi_H}$ depends on the depth $\tilde{r}$ of $\repH{d}$, so there are two cases to consider.
\begin{enumerate}
\item[1)] If $\tilde{r} > r_{d-1}$, we have that $\vec{\phi_H}=(\repH{0},\dots,\repH{d-1},\repH{d})$, and the conclusion follows by what precedes.
\item[2)] If $\tilde{r}\leq r_{d-1}$, we have that  $\vec{\phi_H}=(\repH{0},\dots,\repH{d-1}\repH{d},1)$, so we must check that $\repH{d-1}\repH{d}$ is $H^d$-generic of depth $r_{d-1}$. Because $\tilde{r}\leq r_{d-1}$, this follows from \cite[Lemma 4.9(3)]{Murnaghan:2011}.
\end{enumerate}
\end{proof}

This last proposition completes the proof of Theorem~\ref{th:GderData}. We give a name to the family of $H$-data produced from this theorem in the following definition.

\begin{definition}
Let $\Psi = (\vec{\alg{G}},y,\vec{r}, \rho, \vec{\phi})$ be a $G$-datum, and assume $\rho_H$ decomposes into irreducible representations as $\rho_H = \underset{\ell\in L}{\oplus}\rho_\ell$ for some finite set $L$. For all $\ell\in L$, let ${\Psi_\ell = (\vec{\alg{H}}, y, \vec{r}, \rho_\ell, \vec{\phi}_H)}$ as per Theorem~\ref{th:GderData}. The set $\Res(\Psi) = \{\Psi_\ell: \ell\in L\}$ is called the restriction of $\Psi$. 
\end{definition}

It turns out that one can also define a very natural extension process, using induction, which allows to construct $G$-data from an $H$-datum as per the following theorem.

\begin{theorem}\label{th:backwards}
Let $\Psi_H = (\vec{\alg{H}},y,\vec{r},\rho',\vec{\phi_H})$ be a \emph{normalized} $H$-datum in the sense of \cite[Definition 3.7.1]{Kaletha:2019}, where ${\vec{\alg{H}} = (\alg{H}^0,\dots,\alg{H}^d)}$. Then there exists a (possibly multiple) $G$-datum $\Psi$ such that ${\Psi_H \in\Res(\Psi)}$. 
\end{theorem}

\begin{remark}
There are some $H$-data that do not extend. Indeed, the reason we require $\Psi_H$ to be normalized in the previous theorem is so that $\repH{i}|_{[G^i,G^i]} = 1$ for all $0\leq i\leq d$, which is a necessary condition for extending $\repH{i}$ to a character of $G^i.$ Requiring that we have a normalized $H$-datum is also not very restrictive, as every $H$-datum is $H$-equivalent to a normalized $H$-datum \cite[Lemma 3.7.2]{Kaletha:2019}.
\end{remark}

\begin{proof}[Sketch of the proof of Theorem~\ref{th:backwards}]
We provide the main ideas of the proof for Theorem~\ref{th:backwards} here in the form of a sketch. 

By Theorem~\ref{th:intersectLevi}, there exists a unique twisted Levi sequence $\vec{\alg{G}} = (\alg{G}^0,\dots,\alg{G}^d)$ of $\alg{G}$ such that $\alg{G}^i\cap\alg{H}=\alg{H}^i$ for all $0\leq i\leq d$. Furthermore, $\quo{Z(\alg{G}^0)}{Z(\alg{G})}$ is $F$-anisotropic by what precedes in the proof of Proposition~\ref{prop:D1}.

Using arguments similar to Proposition~\ref{prop:D4}, one can show that any irreducible subrepresentation of $\Ind_{\KH{0}}^{K^0}\rho'$ is $1$-isotypic and compactly induces to a depth-zero supercuspidal representation of $G^0$. Such a representation $\rho$ contains $\rho'$ upon restriction to $\KH{0}$ by construction.

For the extension of the character sequence, the fact that $\repH{i}|_{[G^i,G^i]}=1$ ensures that every irreducible subrepresentation of $\Ind_{H^i}^{G^i}\repH{i}$ is a character of $G^i$ that restricts to $\repH{i}$. Furthermore, any such character is of depth at least $r_i$, as $\repH{i}|_{H^i_{y,r_i}}\neq 1$. One then restricts $\Ind_{H^i}^{G^i}\repH{i}$ to $G^i_{y,r_i^+}$ and, using the fact that $\repH{i}|_{H^i_{y,r_i^+}}=1$, concludes that at least one of the characters in the decomposition has depth $r_i$. Call such a character $\rep{i}$. Similarly to Proposition~\ref{prop:D5}, we find that $\rep{i}$ is automatically $G^{i+1}$-generic of depth $r_i$. 

Letting $\Psi = (\vec{\alg{G}},y,\rho,\vec{\phi})$, where $\vec{\phi} = (\rep{0},\dots,\rep{d})$, we have that $\Psi_H \in \Res(\Psi)$.
\end{proof}

We also give a name to the family of $G$-data produced from the normalized $H$-datum $\Psi_H$ as per the previous theorem.

\begin{definition}
Let $\Psi_H = (\vec{\alg{H}},y,\vec{r},\rho',\vec{\phi_H})$ be a normalized $H$-datum. Let $\mathcal{E}_{\rho'}$ denote the set of irreducible subrepresentations of $\Ind_{\KH{0}}^{K^0}\rho'$ and $\mathcal{E}_i$ denote the set of irreducible subrepresentation of $\Ind_{H^i}^{G^i}\repH{i}$ with depth $r_i, 0\leq i\leq d$. 
The set ${\Ext(\Psi_H) = \{(\vec{G},y,\vec{r},\rho,\vec{\phi}): \rho\in\mathcal{E}_{\rho'}, \rep{i}\in\mathcal{E}_i, 0\leq i\leq d\}}$ is called the extension of $\Psi_H$.
\end{definition}

\subsection{Equivalence of Constructed Data}

We take some time in this section to verify that our restriction and extension of data are compatible with Hakim and Murnaghan's equivalence of data.

\begin{definition}
Given two $G$-data $\Psi$ and $\dot{\Psi}$ of the same length, we will say that ${\Res(\Psi) \simeq \Res(\dot{\Psi})}$ if every element of $\Res(\Psi)$ is $H$-equivalent to an element of $\Res(\dot{\Psi})$ and vice versa.

Analogously, given two normalized $H$-data $\Psi_H$ and $\dot{\Psi}_H$  of the same length, we will say that $\Ext(\Psi_H)\simeq \Ext(\dot{\Psi}_H)$ if every element of $\Ext(\Psi_H)$ is $G$-equivalent to an element of $\Ext(\dot{\Psi}_H)$ and vice versa.
\end{definition}

A natural question that one asks is whether or not the restriction and extension of data described in the previous section are robust. That is, if $\Psi$ and $\dot{\Psi}$ are $G$-equivalent $G$-data, do we have $\Res(\Psi) \simeq \Res(\dot{\Psi})$? Similarly, if $\Psi_H$ and $\dot{\Psi}_H$ are $H$-equivalent normalized $H$-data, do we have $\Ext(\Psi_H) \simeq \Ext(\dot{\Psi}_H)$? 

Given that $H$ is normal in $G$, one can talk about the $G$-conjugation of $H$-data. From the definition, one sees that $\Res({^g\Psi}) = {^g\Res(\Psi)}$ for all $g\in G$. However, when $g \notin H$, we cannot expect to have $\Res(\Psi) \simeq \Res({^g\Psi})$ since equivalence between restrictions is defined at the level of $H$-equivalence (and therefore $H$-conjugation). Similarly, given $g\in G$, we cannot expect $^g\Psi_H$ and $\Psi_H$ to be $H$-equivalent if $g\notin H$, but $\Ext(\Psi_H)$ is clearly stable under $G$-conjugation. That being said, $G$-equivalent $G$-data may not always produce $H$-equivalent restrictions, and $H$-data can produce $G$-equivalent extensions even when they are not $H$-equivalent because of this permitted outer conjugation. As expected, elementary transformations and refactorizations preserve restrictions and extensions. This discussion is summarized by the following two propositions.

\begin{proposition}\label{prop:Hequiv}
Let $\Psi = (\vec{\alg{G}},y,\vec{r},\rho,\vec{\phi})$ and $\dot{\Psi} = (\vec{\dot{\alg{G}}},\dot{y},\vec{r},\dot{\rho},\vec{\dot{\phi}})$ be two $G$-data. Assume one of the three following conditions hold
\begin{enumerate}
\item[1)] $\dot{\Psi}$ is an $H$-conjugate of $\Psi$,
\item[2)] $\dot{\Psi}$ is an elementary transformation of $\Psi$,
\item[3)] $\dot{\Psi}$ is a refactorization of $\Psi$,
\end{enumerate} 
then $\Res(\Psi) \simeq \Res(\dot{\Psi})$.
\end{proposition}

\begin{proof}
1) is immediate. For 2) and 3), one simply needs to take the definitions of elementary transformation \cite[Definition 5.2]{HM:2008} and refactorization \cite[Definition 4.19]{HM:2008} and restrict all relations to the corresponding $H$-subgroups.
\end{proof}

\begin{proposition}\label{prop:GequivConst}
Let $\Psi_H=(\vec{\alg{H}},y,\vec{r},\rho',\vec{\phi_H})$ and $\dot{\Psi}_H = (\vec{\dot{\alg{H}}},\dot{y},\vec{r},\dot{\rho'},\vec{\dot{\phi_H}})$ be two normalized $H$-data. Assume one of the three following conditions hold
\begin{enumerate}
\item[1)] $\dot{\Psi}_H$ is a $G$-conjugate of $\Psi_H$,
\item[2)] $\dot{\Psi}_H$ is an elementary transformation of $\Psi_H$,
\item[3)] $\dot{\Psi}_H$ is a refactorization of $\Psi_H$,
\end{enumerate}
then $\Ext(\Psi_H)\simeq \Ext(\dot{\Psi}_H)$.
\end{proposition}

\begin{proof}
1) is immediate. For 2), it suffices to take the definition of elementary transformation \cite[Definition 5.2]{HM:2008} and to compare $\Ind_{\KH{0}}^{K^0}\rho'$ and $\Ind_{\KH{0}}^{K^0}\dot{\rho'}$. For 3), one fixes an extension $\vec{\phi}$ of $\vec{\phi_H}$ and an element $\rho$ of $\Ind_{\KH{0}}^{K^0}\rho'$ and uses the definition of refactorization \cite[Definition 4.19]{HM:2008} to construct an extension $\vec{\dot{\phi}}$ of $\vec{\dot{\phi_H}}$ and to choose an element $\dot{\rho}\subset\Ind_{\KH{0}}^{K^0}\dot{\rho'}$ that satisfies the wanted relations.
\end{proof}

\section{Restricting the Supercuspidal Representation}\label{sec:restrictSC}

Given the $H$-datum $\Psi_\ell = (\vec{\alg{H}},y,\vec{\tilde{r}}, \rho_\ell,\vec{\phi_H})$ provided by Theorem~\ref{th:GderData}, we can construct open and compact-mod-center subgroups $\KH{i}$ of $H^i$ \`a la J.K.~Yu. Indeed, we define $\KH{0} = H^0_{[y]}$ and $\KH{i+1} = \KH{0}H^1_{y,s_0}\cdots H^{i+1}_{y,s_i}$ for all $0\leq i\leq d-1$. We note the following relationship between $K^i$ and $\KH{i}$.

\begin{proposition}\label{prop:Kintersect}
For all $-1\leq i\leq d-1$, $\KH{i+1} = K^{i+1}\cap H$.
\end{proposition}

\begin{proof}
For $i=-1$, we have $H^0_{[y]} = \{g\in G^0\cap H: g\cdot [y] = [y] \} = G^0_{[y]}\cap H.$

Now, let $0\leq i\leq d-1$. We start by showing that $K^{i+1} = K^0H^1_{y,s_0}\cdots H^{i+1}_{y,s_i}.$
Since $H^{i+1}_{y,s_i}\subset G^{i+1}_{y,s_i}$ for all $0\leq i\leq d-1$, we have $K^0H^1_{y,s_0}\cdots H^{i+1}_{y,s_i}\subset K^{i+1}$. For the converse, let $Z^i = Z(\alg{G}^i)^\circ (F)$ denote the $F$-points of the identity component of the center of $\alg{G}^i$ and let $\overline{G}^i = [\alg{G}^i,\alg{G}^i](F)\subset H^i$. Then $Z^i_{r}\overline{G}^i_{y,r} = G^i_{y,r}$ for all $r >0$ \cite[Lemma B.7.2]{DeBackerReeder:2009}. Applying this result, we obtain
$$
K^{i+1} = K^0G^1_{y,s_0}\cdots G^{i+1}_{y,s_i} = K^0Z^1_{s_0}\overline{G}^1_{y,s_0} \cdots Z^{i+1}_{s_i}\overline{G}^{i+1}_{y,s_i}.$$
We note that $Z^{i+1}_{s_i}\subset Z(\alg{G}^{i+1})(F)\subset Z(\alg{G}^0)(F)$ for all $0\leq i\leq d-1$. Furthermore,  ${Z(\alg{G}^0)(F) \subset K^0}$ as $K^0 = N_{G^0}(G^0_{y,0})$. It follows that
$$
K^{i+1}= K^0\overline{G}^1_{y,s_0}\cdots \overline{G}^{i+1}_{y,s_i} \subset K^0H^1_{y,s_0}\cdots H^{i+1}_{y,s_i},$$ and therefore $K^{i+1} = K^0H^1_{y,s_0}\cdots H^{i+1}_{y,s_i}.$

Now, we clearly have $\KH{i+1} \subset K^{i+1}\cap H$. Conversely, given an element $k'$ of $K^{i+1}\cap H$, it is in particular an element of $K^{i+1}$ and is thus in the form $k'=kh$ for some $k\in K^0, h\in H^1_{y,s_0}\cdots H^{i+1}_{y,s_i} \subset H$ by what precedes. It follows that $k=k'h^{-1}\in K^0\cap H = \KH{0}$. Hence we conclude that $k' \in \KH{i+1}$.
\end{proof}

Analogously to Figure \ref{fig:JKYuConstruction}, we denote by $\primH{i}$ the extension of $\repH{i}$ to $\KH{i+1}$ for all ${0\leq i\leq d-1}$, and by $\kapH{i}$ the inflation of $\primH{i}$ to $\KH{d}$ for all $0\leq i\leq d-2$. We denote the inflation to $\KH{d}$ of $\rho_\ell$ by $\kap{-1}_\ell$. For consistency of notaton, $\kapH{d-1} = \primH{d-1}$ and $\kapH{d} = \repH{d}$. Thus, in the case where the depth $\tilde{r}$ of $\repH{d}$ is greater than $ r_{d-1}$, $\vec{\phi_H} = (\repH{0},\dots,\repH{d-1},\repH{d})$ and $\kappa_H(\Psi_\ell) = \kappa^{-1}_\ell \otimes \kapH{0}\otimes\cdots\otimes\kapH{d-1}\otimes\kapH{d}$. 

In the case where $\tilde{r}\leq r_{d-1}$, then $(\repH{d-1}\repH{d},1)$ are the last two characters of $\vec{\phi_H}$, and therefore $\kappa_H(\Psi_\ell) =\kap{-1}_\ell \otimes \kapH{0} \otimes \cdots \otimes \kapH{d-2} \otimes \kapH{d-1,d} \otimes 1$, where $\kapH{d-1,d} = (\repH{d-1}\repH{d})'$ is the extension of $\repH{d-1}\repH{d}$ to $\KH{d}$.

Since the characters $\repH{i}, 0\leq i\leq d-1$, are simply restricted from the characters $\rep{i}, {0\leq i\leq d-1}$, it is natural to expect that we have $\primH{i} = \prim{i}|_{\KH{i+1}}$ and $\kapH{i} = \kap{i}|_{\KH{d}}$ for all $0\leq i\leq d-1$, which we will prove in Section~\ref{sec:restrictExtInf} with Theorem~\ref{th:phiprime} and Proposition~\ref{prop:kappai}, respectively. This means that the restriction to $H$ of $\pi_G(\Psi)$ commutes in a sense with the steps of the J.K.~Yu Construction, as illustrated in Figure \ref{fig:restrictJKYu}.

As a consequence, in Section~\ref{sec:restrictKapPi} we are able to obtain a formula for $\kappa_G(\Psi)|_{\KH{d}}$ in Theorem~\ref{th:kappa}, which allows us to compute the restriction $\pi_G(\Psi)|_H$ in Theorem~\ref{th:restrictSupercuspidal}.

 \begin{figure}[htbp!]
\center \begin{tikzpicture}
  \matrix (m) [matrix of math nodes,row sep=0.1em,column sep=0.9em,minimum width=0.1em]
  {
    {} &{} &{} & \text{$\left(\rep{i}, K^i\right)$} &{} &\text{$\left(\repH{i}, \KH{i}\right)$} \\
    {} &{} &{} &{} &\text{$\circlearrowleft$} &{} \\
     \text{$\left(\rho, K^0\right)$} &{} &\text{$\left( \underset{\ell\in L}{\oplus}\rho_\ell, \KH{0} \right)$} &\text{$\left(\prim{i}, K^{i+1}\right)$} &{} &\text{$\left(\primH{i}, \KH{i+1}\right)$} \\
      {} &\text{$\circlearrowleft$} &{} &{} &\text{$\circlearrowleft$} &{} \\
     \text{$\left(\kap{-1}, K^d\right)$} &{} &\text{$\left( \underset{\ell\in L}{\oplus}\kap{-1}_\ell, \KH{d}\right)$} & \text{$\left(\kap{i}, K^d\right)$} &{} &\text{$\left( \kapH{i}, \KH{d}\right)$}\\
    };
  \path[-stealth]
     (m-3-1) edge node [above] {restrict} (m-3-3)
   (m-3-1) edge node [left] {inflate} (m-5-1)
    (m-5-1) edge node [below] {restrict} (m-5-3)    
    (m-3-3) edge node [right] {inflate} (m-5-3)
    (m-1-4) edge node [above] {restrict} (m-1-6)
   (m-1-4) edge node [left] {extend} (m-3-4)
    (m-3-4) edge node [below] {restrict} (m-3-6)    
    (m-1-6) edge node [right] {extend} (m-3-6)
    (m-3-4) edge node [left] {inflate} (m-5-4)
    (m-3-6) edge node [right] {inflate} (m-5-6)    
    (m-5-4) edge node [below] {restrict} (m-5-6)
;
\end{tikzpicture}
\caption[Commutativity of restriction with the J.K.~Yu Construction.]{Commutativity of the restriction with extension and inflation for $\rep{i}$, $0\leq i\leq d-1$, and $\rho$.}
\label{fig:restrictJKYu}
\end{figure}

\subsection{Restricting the Extensions $\prim{i}$ and Inflations $\kap{i}$}\label{sec:restrictExtInf}

To talk about the extensions $\primH{i}$, we define the subgroups $\JH{i+1}$ and $\JHP{i+1}$ of $H^{i+1}$ analogously to Section~\ref{sec:construction}.

\begin{align*}
\JH{i+1}(E) = \langle \alg{T}_{\alg{H}}(E)_{r_i}, \alg{U}_a(E)_{y,r_i}, \alg{U}_b(E)_{y,s_i} : a\in\Phi(\alg{G}^i,\alg{T}), b\in \Phi(\alg{G}^{i+1},\alg{T}) \setminus \Phi(\alg{G}^i,\alg{T}) \rangle, \nonumber \\
\JHP{i+1}(E) = \langle \alg{T}_{\alg{H}}(E)_{r_i}, \alg{U}_a(E)_{y,r_i}, \alg{U}_b(E)_{y,s_i^+} : a\in\Phi(\alg{G}^i,\alg{T}), b\in \Phi(\alg{G}^{i+1},\alg{T}) \setminus \Phi(\alg{G}^i,\alg{T}) \rangle.
\end{align*}

We write $\JH{i+1}$ and $\JHP{i+1}$ for $\JH{i+1}(E)\cap H^{i+1}$ and $\JHP{i+1}(E)\cap H^{i+1}$, respectively. We have that $\JH{i+1} = (H^i,H^{i+1})_{y,(r_i,s_i)}, \JHP{i+1} = (H^i,H^{i+1})_{y,(r_i,s_i^+)}$ and $\KH{i+1} = \KH{i}\JH{i+1}$ for all $0\leq i\leq d-1$.

\begin{proposition}
For all $0\leq i\leq d-1$, $J^{i+1} = \JH{i+1}J^{i+1}_+$.
\end{proposition}

\begin{proof}
First, recall that $\alg{T}(E)_{r_i}$ normalizes the filtrations of the root subgroups that partly generate $J^{i+1}(E)$. So, given $g\in J^{i+1}(E)$, we can write $g = ut$ for some $t\in \alg{T}(E)_{r_i} \subset J^{i+1}_+(E)$ and $$u\in \langle \alg{U}_a(E)_{y,r_i},\alg{U}_b(E)_{y,s_i}: a\in \Phi(\alg{G}^i,\alg{T}), b\in\Phi(\alg{G}^{i+1},\alg{T})\setminus \Phi(\alg{G}^i,\alg{T}) \rangle \subset \JH{i+1}(E).$$ Therefore, we have $J^{i+1}(E) = \JH{i+1}(E)J^{i+1}_+(E)$.

We proceed similarly to the proof of \cite[Lemma 2.10]{Yu:2001} to show that the equality also holds at the level of $F$-points. Take ${g\in J^{i+1} = J^{i+1}(E)\cap G^{i+1}}$. Then $g$ is fixed by $\Gal(E/F)$. By what precedes, $g= ut$ for some $u \in \JH{i+1}(E)$, $t \in J^{i+1}_+(E)$. It follows that the image of $g$ in ${\quo{J^{i+1}(E)}{J^{i+1}_+(E)} = \quo{\JH{i+1}(E)J^{i+1}_+(E)}{J^{i+1}_+(E)}}$, given by $uJ^{i+1}_+(E)$, is fixed by $\Gal(E/F)$. But the second isomorphism theorem tells us ${\quo{\JH{i+1}(E)J^{i+1}_+(E)}{J^{i+1}_+(E)} \simeq \quo{\JH{i+1}(E)}{\JH{i+1}(E)\cap J^{i+1}_+(E)}}$, where ${\JH{i+1}(E)\cap J^{i+1}_+(E) = \JHP{i+1}(E)}$. Hence, $u \JHP{i+1}(E)$ is fixed by $\Gal(E/F)$.

Next, \cite[Corollary 2.3]{Yu:2001} states that the natural homomorphism
$$\JH{i+1} \rightarrow \left( \quo{\JH{i+1}(E)}{\JHP{i+1}(E)} \right)^{\Gal(E/F)}$$ is a surjection. 
This implies that there exists $g_0 \in \JH{i+1}$ such that $g_0\JHP{i+1}(E) = u\JHP{i+1}(E)$, that is $g_0^{-1}u \in \JHP{i+1}(E)$. Hence, $u \in \JH{i+1}\JHP{i+1}(E)$. We deduce that $g = g'\overline{g}$ for some ${g' \in \JH{i+1}}, {\overline{g}\in J^{i+1}_+(E)}$. Furthermore, we have that $\overline{g} = g'^{-1}g \in G^{i+1}$, and therefore $\overline{g}\in J^{i+1}_+$. The result follows.
\end{proof}

\begin{corollary}\label{cor:beta}
Let $\beta : \quo{\JH{i+1}}{\JHP{i+1}} \rightarrow \quo{J^{i+1}}{J^{i+1}_+}$ be defined by $\beta(j\JHP{i+1}) = jJ^{i+1}_+$ for all $j\in \JH{i+1}$. Then $\beta$ is an isomorphism.
\end{corollary}

\begin{remark}\label{rem:JJ+}
As a consequence of the last corollary, we have that $J^{i+1} = J^{i+1}_+$ if and only if $\JH{i+1} = \JHP{i+1}$ for all $0\leq i\leq d-1$. This is an important result as it implies that the theory of Weil representations is always applied simultaneously in the construction of $\prim{i}$ and $\primH{i}$. This is crucial for us to be able to relate those two extensions. Indeed, if it was possible for $\JH{i+1} = \JHP{i+1}$, but $J^{i+1}\neq J^{i+1}_+$, then talking about restricting $\prim{i}$, a representation of dimension greater than 1, would not result in the character $\primH{i}$.
\end{remark}

\begin{theorem}\label{th:phiprime}
For all $0\leq i\leq d-1$, $\primH{i} = \prim{i}|_{\KH{i+1}}$.
\end{theorem}

\begin{proof}
For a fixed $0\leq i\leq d-1$, we have two cases to consider.

\textit{Case 1 ($J^{i+1} = J^{i+1}_+$):} In this case, we have $\JH{i+1} = \JHP{i+1}$ by Remark~\ref{rem:JJ+}. It follows that $\primH{i}(kj) = \repH{i}(k)\widehat{\repH{i}}(j)$ for all $k\in \KH{i}, j\in \JH{i+1}$, where $\widehat{\repH{i}}$ is the (unique) quasicharacter of $\KH{i}H^{i+1}_{y,s_i^+}$ that agrees with $\repH{i}$ on $\KH{i}$ and is trivial on $(H^i,H^{i+1})_{y,(r_i^+,s_i^+)}$.

Clearly, we have that $\widehat{\rep{i}} |_{\KH{i}H^{i+1}_{y,s_i +}}$ agrees with $\repH{i}$ on $\KH{i}$, as it agrees with $\rep{i}$ on $K^i$, and is trivial on $(H^i,H^{i+1})_{y,(r_i^+,s_i^+)}$, as it is trivial on $(G^i,G^{i+1})_{y,(r_i^+,s_i^+)}$. Therefore, by uniqueness, we have $\widehat{\rep{i}} |_{\KH{i}H^{i+1}_{y,s_i +}} = \widehat{\repH{i}}$.

So, for all $k\in \KH{i}, j\in \JH{i+1}$
$$\prim{i}|_{\KH{i+1}}(kj) = \rep{i}(k)\widehat{\rep{i}}(j) =\repH{i}(k)\widehat{\repH{i}}(j) = \primH{i}(kj).$$

\textit{Case 2 ($J^{i+1}\neq J^{i+1}_+$):} In this case, we have that $\prim{i}$ is constructed using the Heisenberg-Weil lift $\weil{i}$, which is a representation of $K^i\ltimes \Heis{i}$, where $\Heis{i} = \quo{J^{i+1}}{\ker(\xiG{i})}$ and $\xiG{i} = \widehat{\rep{i}}|_{J^{i+1}_+}$. Since $\JH{i+1} \neq \JHP{i+1}$, we require a Heisenberg-Weil lift to extend $\widehat{\repH{i}}$, which we denote by $\weilH{i}$. Note that $\weilH{i}$ is a representation of $\KH{i} \ltimes \HeisH{i}$, where ${\HeisH{i} = \quo{\JH{i+1}}{\ker(\xiH{i})}}$ and ${\xiH{i}= \widehat{\repH{i}}|_{\JHP{i+1}} = \widehat{\rep{i}}|_{\JHP{i+1}}}$. We recall that in this case, for all $k\in \KH{i}, j\in \JH{i+1}$, $\prim{i}|_{\KH{i+1}}(kj) = \widehat{\rep{i}}(k)\weil{i}(k,j\ker(\xiG{i}))$ and $\primH{i}(kj) = \widehat{\repH{i}}(k)\weilH{i}(k,j\ker(\xiH{i}))$. By what precedes, $\widehat{\rep{i}}(k) = \widehat{\repH{i}}(k)$ for all $k\in \KH{i}$. Therefore, to show that $\prim{i}|_{\KH{i+1}} = \primH{i}$, we must show that $\weil{i}(k,j\ker(\xiG{i})) = \weilH{i}(k,j\ker(\xiH{i}))$ for all $k\in\KH{i}, j\in \JH{i+1}$.

Let $f^i: K^i\rightarrow\Sp(\Heis{i})$ and $f^i_H: \KH{i}\rightarrow \Sp(\HeisH{i})$ be the homomorphisms coming from the actions by conjugation of $K^i$ on $J^{i+1}$ and of $\KH{i}$ on $\JH{i+1}$, respectively. Recall the symplectic vector space $W^i = \quo{J^{i+1}}{J^{i+1}_+}$ and define analogously $\WH{i} = \quo{\JH{i+1}}{\JHP{i+1}}$. Let $\spec{i}$ and $\specH{i}$ be the special isomorphisms arising from the split polarizations of $\Heis{i}$ and $\HeisH{i}$ which map into ${\W{i}}^\# = \W{i}\boxtimes \mathbb{F}_p$ and ${\WH{i}}^\# = \WH{i}\boxtimes \mathbb{F}_p$, respectively, as per \cite[Lemma 2.35]{HM:2008}. One can verify that the homomorphism $\alpha: \HeisH{i}\rightarrow \Heis{i}$ which maps $j\ker(\xiH{i})$ to $j\ker(\xiG{i})$ for all $j\in\JH{i+1}$ is an isomorphism that induces the symplectic isomorphism $\beta : \WH{i} \rightarrow \W{i}$ from Corollary~\ref{cor:beta}. Let $\delta: \KH{i} \rightarrow K^i$ be the inclusion map. Finally, let $(f^i)_{\spec{i}}$ be the map that sends $k\in K^i$ to $\spec{i}\circ f^i(k) \circ {\spec{i}}^{-1}$ and define $(f^i_H)_{\specH{i}}$ similarly. It is then straightfoward to verify that the two diagrams in Figure~\ref{fig:commutativeDiagram} commute.

\begin{figure}[!htbp]
\center
\begin{tikzpicture}
  \matrix (m) [matrix of math nodes,row sep=3em,column sep=4em,minimum width=2em]
  {
     \text{$\HeisH{i}$} & \text{${\WH{i}}^\#$} & &\text{$\KH{i}$} &\text{$Sp(\WH{i})$} \\
     \text{$\Heis{i}$} &\text{${\W{i}}^\#$}  & &\text{$K^i$} &\text{$Sp(\W{i})$}\\};
  \path[-stealth]
    (m-1-1) edge node [left] {$\alpha$} (m-2-1)
            edge node [above] {$\specH{i}$} (m-1-2)
    (m-2-1) edge node [above] {$\spec{i}$} (m-2-2)
    (m-1-2) edge node [right] {$\beta\times 1$}(m-2-2)
    (m-1-4) edge node [left] {$\delta$} (m-2-4)
    			edge node [above] {$(f^i_H)_{\specH{i}}$} (m-1-5)
    	(m-2-4) edge node [above] {$(f^i)_{\spec{i}}$} (m-2-5)
    	(m-1-5) edge node [right] {$\inn(\beta)$} (m-2-5);
\end{tikzpicture}
\caption{Commutative diagrams for the restriction of the Heisenberg-Weil lift.}
\label{fig:commutativeDiagram}
\end{figure}

Applying \cite[Proposition 3.2]{Nevins:2015}, we have that $\weilH{i}(k,j\ker(\xiH{i})) = \weil{i}(\delta(k),\alpha(j\ker(\xiH{i})))$ for all $k\in\KH{i},j\in\JH{i+1}$ and the result follows.

\end{proof}

\begin{proposition}\label{prop:kappai}
For all $0\leq i\leq d-1, \kapH{i} = \kap{i}|_{\KH{d}}$. Furthermore, $\kap{-1}|_{\KH{d}} = \underset{\ell \in L}{\oplus} \kap{-1}_\ell.$
\end{proposition}

\begin{proof}
Because $\primH{i} = \prim{i}|_{\KH{i+1}}$ for all $0\leq i\leq d-1$ (Theorem~\ref{th:phiprime}), it easily follows from the definition of inflation that $\kapH{i} =  \kap{i}|_{\KH{d}}.$

For all $k\in\KH{0}, j\in J^1_H\cdots J^d_H$, we have $\kap{-1}(kj) = \rho(k) = \underset{\ell \in L}{\oplus}\rho_\ell(k) = \underset{\ell\in L}{\oplus}\kappa^{-1}_\ell(kj)$, which proves the second statement.
\end{proof}

\subsection{Restricting $\kappa_G(\Psi)$ and $\pi_G(\Psi)$}\label{sec:restrictKapPi}

Now that we have calculated the restrictions of $\prim{i}$ and $\kap{i}$ for all $0\leq i\leq d$, we can proceed to computing the restriction of $\kappa_G(\Psi)$.

\begin{theorem}\label{th:kappa}
Given the $G$-datum $\Psi$, $\kappa_G(\Psi)|_{\KH{d}} \simeq \underset{\ell\in L}{\oplus}\kappa_H(\Psi_\ell)$.
\end{theorem}

\begin{proof}
We have that
$\kappa_G(\Psi)|_{\KH{d}} = \kap{-1}|_{\KH{d}} \otimes \underset{0 \leq i \leq d} {\bigotimes}\kap{i} |_{\KH{d}}.$
 From Proposition~\ref{prop:kappai}, we can rewrite the right-hand side as
\begin{align*}
\kappa_G(\Psi)|_{\KH{d}} 
 = \left(\underset{\ell\in L}{\oplus} \kap{-1}_\ell \right) \otimes \underset{0 \leq i \leq d} {\bigotimes}\kapH{i}
 = \underset{\ell\in L}{\oplus}\left( \kap{-1}_\ell\otimes \underset{0 \leq i \leq d}{\bigotimes}\kapH{i} \right). 
 \end{align*}

In the case where the depth $\tilde{r}$ of $\repH{d}$ is greater than $r_{d-1}$, then ${\kappa_H(\Psi_\ell)=\kap{-1}_\ell\otimes \underset{0 \leq i \leq d}{\bigotimes}\kapH{i}}$ and we obtain our result. When $\tilde{r}\leq r_{d-1}$, we must find a way to compare $\kapH{d-1}\otimes\kapH{d}$ with $\kapH{d-1,d}\otimes 1$, where $\kapH{d-1,d} = (\repH{d-1}\repH{d})'$. 

In \cite[Proposition 4.24]{HM:2008}, Hakim and Murnaghan show that if $\dot{\Psi}_H$ is a refactorization of $\Psi_H$, then $\kappa_H({\Psi}_H)\simeq \kappa_H(\dot{\Psi}_H)$. When looking at the details of their proof, we only need $\vec{\phi}_H$ and $\vec{\dot{\phi}}_H$ to satisfy condition F1) from the definition of refactorization \cite[Definition 4.19]{HM:2008} for this result to hold, and we do not even require the last quasicharacters of the sequences to be of largest depth. This means that we can set $\vec{\phi}_H = (\repH{0},\dots,\repH{d-1},\repH{d})$ and $\vec{\dot{\phi}}_H = (\repH{0},\dots,\repH{d-1}\repH{d},1)$ and obtain $\kapH{d-1}\otimes\kap{d}\simeq\kapH{d-1,d}\otimes 1$. Thus, 
$$\kap{-1}_\ell\otimes \underset{0 \leq i \leq d}{\bigotimes}\kapH{i} \simeq \kap{-1}_\ell \otimes \kapH{0}\otimes\cdots\otimes \kapH{d-2}\otimes\kapH{d-1,d}\otimes 1 = \kappa_H(\Psi_\ell),$$ and the result follows.

\end{proof}

We are now ready to calculate the restriction of $\pi_G(\Psi)$ to $H$.

\begin{theorem}\label{th:restrictSupercuspidal}
Given the datum $\Psi$, $\pi_G(\Psi)|_H \simeq\underset{t\in C}{\oplus} \left( \underset{\ell\in L}{\oplus} \pi_H({^t\Psi_\ell})\right)$, where $C$ is a set of coset representatives of $\Bigslant{H}{G}{K^d}$. 
\end{theorem}

\begin{proof} 
From the Mackey Decomposition, we have that
$$\Res_H^G \pi_G(\Psi) = \Res_H^G\Ind_{K^d}^G\kappa_G(\Psi)
= \underset{t\in C}{\oplus} \Ind_{^t(K^d\cap H)}^{H}\Res_{^t(K^d \cap H)}^{^t(K^d)}{^t\kappa_G(\Psi)}.$$ We have that $^t\kappa_G(\Psi)= \kappa_G({^t\Psi})$ \cite[Section 5.1.1]{HM:2008}, which implies that 
$$\Res_H^G \pi_G(\Psi) =  \underset{t\in C}{\oplus} \Ind_{^t\KH{d}}^H\Res_{^t\KH{d}}^{^tK^d}\kappa_G(^t\Psi).$$
Applying Theorem~\ref{th:kappa} to the datum $^t\Psi$, we have that ${\Res_{^t\KH{d}}^{^tK^d}\kappa_G(^t\Psi) \simeq \underset{\ell\in L}{\oplus}\kappa_H(({^t\Psi})_\ell)}$. From the definition of conjugate $G$-datum, it is clear that $({^t\Psi})_\ell = {^t\Psi_\ell}$. Finally, given that induction commutes with direct sum, we obtain
$$\Res_H^G \pi_G(\Psi)  = \underset{t\in C}{\oplus}\left( \underset{\ell\in L}{\oplus} \Ind_{^t\KH{d}}^H\kappa_H({^t\Psi}_\ell) \right) = \underset{t\in C}{\oplus}\left( \underset{\ell\in L}{\oplus} \pi_H({^t\Psi_\ell}) \right).$$
\end{proof}

Using Frobenius Reciprocity and the exhaustion of the J.K.~Yu Construction, one can obtain an analogous result concerning the induction to $G$ of a supercuspidal representation of $H$.

\begin{theorem}\label{th:induceSupercuspidal}
Let $\Psi_H$ be a normalized $H$-datum. Every irreducible supercuspidal subrepresentation of $\Ind_H^G\pi_H(\Psi_H)$ is of the form $\pi_G(\Psi)$, where $\Psi\in {^{t^{-1}}\Ext(\Psi_H})$ for some $t\in C$, $C$ a set of coset representatives of $\Bigslant{H}{G}{K^d}$.
\end{theorem}

\section{Multiplicity in the Restriction}\label{sec:multiplicity}

Given that we have a precise description of $\pi_G(\Psi)|_H$ provided by Theorem~\ref{th:restrictSupercuspidal}, we determine the multiplicity of each component of this restriction in this section. We start by establishing that the source of the multiplicity relies on $\rho_H=\rho|_{H^0_{[y]}}$, and then compute an explicit formula for the multiplicity in the restriction. 

\subsection{Source of Multiplicity}

\begin{theorem}\label{th:multiplicity}
Let $C$ be a set of coset representatives of $\Bigslant{H}{G}{K^d}$ and let $\ell\in L$. Then for all $t\in C$, the multiplicity of the component $\pi_H(^t\Psi_\ell)$ in $\pi_G(\Psi)|_H$ is equal to the multiplicity of $\rho_\ell$ in $\rho_H$. In particular, $\pi_G(\Psi)|_H$ is multiplicity free if and only if $\rho_H$ is multiplicity free.
\end{theorem}

The proof of this theorem will be divided into two lemmas (Lemmas~\ref{lem:notconjugate} and \ref{lem:rhoiso}).

\begin{lemma}\label{lem:notconjugate}
Let $C$ be a set of representatives of $\Bigslant{H}{G}{K^d}$, and let $s,t\in C$. Assume that $s\neq t$. Then for all $\ell,k\in L$, $\pi_H({^s\Psi_\ell}) \not\simeq \pi_H({^t\Psi_k})$.
\end{lemma}

\begin{proof}
We show the contrapositive statement. Assume that $\pi_H({^s\Psi_\ell}) \simeq \pi_H({^t\Psi_k})$. From Theorem~\ref{th:HM67}, there exists $h\in H$ such that $\KH{0} = {^{s^{-1}ht}{\KH{0}}}$, $\vec{\alg{H}} = {^{s^{-1}ht}\vec{\alg{H}}}$ and ${\rho_\ell\otimes \phi_H \simeq {^{s^{-1}ht}(\rho_k\otimes \phi_H)}}.$ 

Since $\rho_j|_{H^0_{y,0^+}}$ is $1$-isotypic for $j=\ell,k$, it follows that 
$\phi_H|_{H^0_{y,0^+}} = ({^{s^{-1}ht}\phi_H})|_{H^0_{y,0^+}}.$ Given that the depth of a character does not depend on the point of the building that describes the Moy-Prasad filtration subgroups \cite[Definition 2.46]{HM:2008}, we have that
$\phi_H|_{H^0_{0^+}} = ({^{s^{-1}ht}\phi_H})|_{H^0_{0^+}}$, where $H^0_{0^+} = \underset{x\in \mathcal{B}(\alg{H}^0,F)}{\bigcup}H^0_{x,0^+}$.
By \cite[Proposition 5.4]{Murnaghan:2011}, we have $\prod\limits_{j=i}^d\repH{j}|_{H^i_{r_{i-1}^+}} = \prod\limits_{j=i}^d{^{s^{-1}ht}\repH{j}}|_{H^i_{r_{i-1}^+}}$ for $0\leq i\leq d$. Letting $Z^i = Z(\alg{G}^i)^\circ(F)$ and $\overline{G}^i = [\alg{G}^i,\alg{G}^i](F)\subset H^i$, we have that $G^i_{r^+} = Z^i_{r^+}\overline{G}^i_{r^+}$ for all $r\geq 0$ \cite[Lemma B.7.2]{DeBackerReeder:2009} and therefore $G^i_{r^+} = Z^i_{r^+}H^i_{r^+}$. As a consequence, we have that $G^i_{r_{i-1}^+} = {^{s^{-1}ht}G^i_{r_{i-1}^+}}$ and 
$\prod\limits_{j=i}^d\rep{j}|_{G^i_{r_{i-1}^+}} = \prod\limits_{j=i}^d{^{s^{-1}ht}\rep{j}}|_{G^i_{r_{i-1}^+}} \text{ for } 0\leq i\leq d.$  Using an argument similar to \cite[Lemma 5.6]{Murnaghan:2011}, we conclude that $s^{-1}ht \in G^0$.

Furthermore, having $^{s^{-1}ht}\KH{0} = \KH{0}$ implies $s^{-1}ht\cdot [y] = [y]$ \cite[Proof of Theorem 6.7]{HM:2008}, and therefore ${s^{-1}ht\in G^0_{[y]}} = K^0 \subset K^d$. 
Hence, $t\in HsK^d$, meaning that $s$ and $t$ belong to the same double coset of $\Bigslant{H}{G}{K^d}$. Because $s$ and $t$ are coset representatives, we conclude that they must be equal.
\end{proof}

The previous lemma implies that the multiplicity in the restriction $\pi_G(\Psi)|_H$ must come from within $\underset{\ell\in L}{\oplus}\pi_H(\Psi_\ell)$. The next step is then to determine the condition under which $\pi_H(\Psi_\ell) \simeq \pi_H(\Psi_k)$, $\ell,k\in L$, which is provided by the following lemma.

\begin{lemma}\label{lem:rhoiso}
Let $\ell, k\in L$. Then $\pi_H(\Psi_\ell) \simeq \pi_H(\Psi_k)$ if and only if $\rho_\ell \simeq \rho_k$.
\end{lemma}

\begin{proof}
Assume that $\rho_\ell \simeq \rho_k$. Then $\Psi_k$ is an elementary transformation \cite[Definition 5.2]{HM:2008} of $\Psi_\ell$ and $\pi_H(\Psi_\ell)\simeq \pi_H(\Psi_k)$ by Theorem~\ref{th:HM67}. 

Conversely, if $\pi_H(\Psi_\ell)\simeq \pi_H(\Psi_k)$, we know from Theorem~\ref{th:HM67} that there exists $h\in H$ such that
$\vec{\alg{H}} = {^h\vec{\alg{H}}}, \KH{0} = {^h\KH{0}}$, and $\rho_\ell \otimes \phi_H \simeq {^h(\rho_k \otimes \phi_H)}$, where ${\phi_H = \prod\limits_{i=0}^d\repH{i}|_{H^0}}$. 
Proceeding similarly to the proof of Lemma~\ref{lem:notconjugate}, we have that $h\in K^0\cap H = \KH{0}$. Because $h\in H^0$, we have $\phi_H = {^h\phi_H}$. We then conclude that $\rho_\ell \simeq {^h\rho_k}$, which implies that $\rho_\ell \simeq \rho_k$ as $h\in \KH{0}$.
\end{proof}

Combining Lemmas~\ref{lem:notconjugate} and \ref{lem:rhoiso} together gives us the proof of Theorem~\ref{th:multiplicity}.

\begin{proof}[Proof of Theorem~\ref{th:multiplicity}]
Let $s\in C$ and $j\in L$. By Lemma~\ref{lem:rhoiso}, the multiplicity of $\rho_j$ in $\rho_H$ corresponds to the multiplicity of $\pi_H({^s\Psi_j})$ in $\underset{\ell\in L}{\oplus}\pi_H({^s\Psi_\ell})$. Since $\pi_H({^s\Psi_j}) \not\simeq \pi_H({^t\Psi_k})$ when $s\neq t$ regardless of $k\in L$ (Lemma~\ref{lem:notconjugate}), this multiplicity corresponds to the multiplicity of $\pi_H({^s\Psi_j})$ in  $\underset{t\in C}{\oplus} \left( \underset{\ell\in L}{\oplus} \pi_H({^t\Psi_\ell})\right)$.
\end{proof}

Theorems \ref{th:restrictSupercuspidal} and \ref{th:multiplicity} generalize \cite[Theorem 5.2]{Nevins:2015} in which $H = \Gder$ and $\Psi$ is a toral datum of length one. In this case $\rho$ is assumed trivial and $\Res(\Psi)$ consists of only one element (Corollary \ref{prop:NevinsDatum}), meaning that the restriction $\pi_G(\Psi)|_H$ is automatically multiplicity free. 


We end this section by mentioning that Theorem~\ref{th:multiplicity} applies to both positive-depth and depth-zero supercuspidal representations as they are all produced from the J.K.~Yu Construction. This allows us to restate our result as follows.

\begin{corollary}\label{cor:multiplicityDepthZero}
Let $\ell\in L$. Then the multiplicity of $\pi_H(\Psi_\ell)$ in $\pi_G(\Psi)|_H$ is equal to the multiplicity of $\pi^{-1}_{\ell}$ in $\pi^{-1}|_{H^0}$, where $\pi^{-1} \defeq \Ind_{K^0}^{G^0}\rho$ and $\pi^{-1}_\ell \defeq \Ind_{\KH{0}}^{H^0}\rho_\ell$. In particular, $\pi_G(\Psi)|_H$ is multiplicity free if and only if $\pi^{-1}|_{H^0}$ is multiplicity free. 
\end{corollary}

\begin{proof}
We know from Theorem~\ref{th:multiplicity} that the multiplicity of $\pi_H(\Psi_\ell)$ in $\pi_G(\Psi)|_H$ is equal to the multiplicity of $\rho_\ell$ in $\rho_H$. On the other hand, if we consider the datum $\Psi^0=(G^0,y,\rho,1)$, we have $\pi_{G^0}(\Psi^0) = \pi^{-1}$. Then, Theorem~\ref{th:multiplicity} tells us that the multiplicity of $\pi_{H^0}(\Psi^0_\ell)=\pi^{-1}_\ell$ in $\pi_{G^0}(\Psi^0)|_{H^0}$ is equal to the multiplicity of $\rho_\ell$ in $\rho_H$. The result thus follows.
\end{proof}

The statement of Corollary~\ref{cor:multiplicityDepthZero} makes it clear that the multiplicity upon restriction of a positive-depth supercuspidal representation really depends on the depth-zero supercuspidal representation of the (reduced) datum. Therefore, the question of determining multiplicities in the restriction of an irreducible supercuspidal representation reduces to an equivalent question for irreducible depth-zero supercuspidal representations.

\subsection{Computing the Multiplicity}

Theorem~\ref{th:multiplicity} implies that we must study the decomposition of $\rho_H = \rho|_{\KH{0}}$ in order to compute the multiplicities in $\pi_G(\Psi)|_H$. Because $\KH{0}$ is a normal subgroup of $K^0$, the strategy is to apply Clifford theory to obtain a description for $\rho_H$. 

\begin{definition}
Let $\pi$ be an irreducible representation of a group $G'\subset G$. The set ${I_G(\pi) = \{g\in G : {^g\pi}\simeq \pi\}}$ is called the normalizer of $\pi$ in $G$, or the inertial subgroup of $\pi$ (in $G$).
\end{definition}

\begin{proposition}\label{prop:multiplicityrhoH}
Let $\rho'$ be an irreducible subrepresentation of $\rho|_{\KH{0}}$, $I = I_{K^0}(\rho')$ and denote the dimensions of $\rho$ and $\rho'$ by $\dim(\rho)$ and $\dim(\rho')$, respectively. Then $\rho|_{\KH{0}} = m\left(\underset{c\in C}{\oplus} {^c\rho'} \right)$, where $$m = \frac{\dim(\rho)}{[K^0:I]\dim(\rho')}$$ and $C$ is a set of coset representatives of $\quo{K^0}{I}$ Furthermore, every component of $\rho|_{\KH{0}}$ has multiplicity $m$. 
\end{proposition}

\begin{proof}
Let $Z^0$ denote the center of $G^0$. We first extend the representation $\rho'$ to a representation $\overline{\rho'}$ of $Z^0\KH{0}$ acting on the same vector space. This can be done easily, as $Z^0$ acts via a character. This action of $Z^0$ implies that $^g\rho' \simeq \rho'$ if and only if $^g\overline{\rho'}\simeq\overline{\rho'}$ so that $I = I_{K^0}(\overline{\rho'})$.

Because $Z^0\KH{0}$ is a normal subgroup of finite index in $K^0$ and that $\rho$ is finite-dimensional, one can easily adapt the results from Clifford Theory for finite groups. By \cite[Lemma 2.1(2)]{Clifford}, we have that $\rho = \Ind_I^{K^0}\Pi$ for some irreducible subrepresentation ${\Pi\subset \Ind_{Z^0\KH{0}}^{I}\overline{\rho'}}$. Using the Mackey Decomposition on this last inclusion, one sees that $\Pi|_{Z^0\KH{0}}$ is $\overline{\rho'}$-isotypic (and therefore $\Pi|_{\KH{0}}$ is $\rho'$-isotypic) of a certain multiplicity, say $m$. We have that ${\dim(\Pi) = m\dim(\rho')}$ and $\dim(\rho) = [K^0:I]\dim(\Pi)$, and therefore $$m = \frac{\dim(\rho)}{[K^0:I]\dim(\rho')}.$$

One then applies the Mackey Decomposition on $\rho$ to obtain
$$\rho|_{\KH{0}} = \Res_{\KH{0}}^{K^0}\Ind_{I}^{K^0}\Pi = \underset{c\in C}{\oplus} \Ind^{\KH{0}}_{\KH{0}\cap {^c I}}\Res^{^c I}_{\KH{0}\cap {^c I}}{^c\Pi},$$ where $C$ is a set of coset representatives of $\Bigslant{\KH{0}}{K^0}{I}$. Because $\KH{0}$ is a normal subgroup of $K^0$ and $\KH{0} \subset I$, $\Bigslant{\KH{0}}{K^0}{I} = \quo{K^0}{I}$ and
$$\rho|_{\KH{0}} = \underset{c\in C}{\oplus} \Res_{\KH{0}}^{^c{I}}{^c\Pi} =  m \left(\underset{c\in C}{\oplus} {^c{\rho'}} \right).$$
\end{proof}

As a corollary of Proposition~\ref{prop:multiplicityrhoH} and Theorem~\ref{th:multiplicity}, we can refine our description of $\pi_G(\Psi)|_H$ from Theorem~\ref{th:restrictSupercuspidal}.

\begin{corollary}\label{cor:refineRes}
Let $\Psi = (\vec{\alg{G}},y,\rho,\vec{\phi})$ be a $G$-datum, $\rho'$ be a component of $\rho_H$ and $I = I_{K^0}(\rho')$. Let $D$ be a set of coset representatives of $\quo{K^0}{I}$, and for all $d\in D$ let $\Psi_d$ be the $H$-datum associated to $^d\rho'$. Then
$$\pi_G(\Psi)|_H = m \left( \underset{t\in C, d\in D}{\bigoplus} \pi_H({^t\Psi_d}) \right),$$ where $$m = \frac{\dim(\rho)}{[K^0:I]\dim(\rho')},$$ and $C$ is a set of coset representatives of $\Bigslant{H}{G}{K^d}$. Furthermore, every component of $\underset{t\in C, d\in D}{\bigoplus} \pi_H({^t\Psi_d})$ is distinct, and every component of $\pi_G(\Psi)|_H$ has multiplicity $m$.
\end{corollary}

\section{Computing Multiplicity for Regular Depth-Zero Supercuspidal Representations}\label{sec:regularDZ}
By what precedes in Corollary~\ref{cor:refineRes}, when one can single out a component of $\rho_H$ and compute its normalizer in $K^0$, the multiplicity $m$ of each component of $\pi_G(\Psi)|_H$ can be explicitly calculated. 
The goal of this section is to provide an example of a class of representations $\rho$ for which it is possible to explicitly isolate components $\rho'$ in $\rho_H$ and derive conditions for a multiplicity free restriction.

Because $\alg{H}^0$ is such that $[\alg{G}^0,\alg{G}^0] \subset \alg{H}^0$, we can alleviate notation and assume that $\rho$ is an irreducible representation of $G_{[y]}$ that induces irreducibly to a depth-zero supercuspidal representation of $G$ that we wish to restrict to $H_{[y]}$ (that is replace $\alg{G}^0$ by $\alg{G}$ and $\alg{H}^0$ by $\alg{H}$). 
What we know about $\rho$ is that $\rho|_{G_{y,0}}$ contains the pullback of a cuspidal representation $\sigma$ of $G_{y,0:0^+}$ \cite[Theorem 6.8]{MP:1996}, where  $G_{y,0:0^+} = \mathcal{G}_y(\res)$.
However, we need more information to compute its dimension and isolate one of the components of $\rho_H = \rho|_{H_{[y]}}$ to compute $m$.

For this reason, we restrict our attention to when $\rho$ induces to a \emph{regular} depth-zero supercuspidal representation in the sense of \cite[Definition 3.4.19]{Kaletha:2019}, meaning that $\rho|_{G_{y,0}}$ contains a Deligne-Lusztig cuspidal representation $\pm R_{\mathcal{S},\overline{\theta}}$ satisfying an additional regularity condition. In the Appendix, we provide some calculations related to Deligne-Lusztig virtual characters which complement the results of this section. The regular depth-zero supercuspidal representations are a very important class. Indeed, Kaletha points out that most depth-zero supercuspidal representations are regular (paragraph following \cite[Definition 3.7.3]{Kaletha:2019}).

The regular depth-zero supercuspidal representations of $G$ are constructed from pairs $(\alg{S},\theta)$ \cite[Proposition 3.4.27]{Kaletha:2019}, where $\alg{S}$ is a maximally unramified elliptic maximal $F$-torus of $\alg{G}$ in the sense of \cite[Definition 3.4.2]{Kaletha:2019} and $\theta$ is a regular depth-zero character of $S$, meaning that $I_{N_\alg{G}(\alg{S})(F)}(\theta|_{S_0}) = S$ \cite[Definition 3.4.16]{Kaletha:2019}. Kaletha's construction is summarized as follows and illustrated in Figure~\ref{fig:SummaryKaletha}.

Given a maximally unramified elliptic maximal $F$-torus $\alg{S}$ of $\alg{G}$, one can associate a vertex $[y]$ of $\mathcal{B}^\mathrm{red}(\alg{G},F)$ \cite[Lemma 3.4.3]{Kaletha:2019}, which is the unique $\Gal(F^\un/F)$-fixed point of $\mathcal{A}^{\mathrm{red}}(\alg{G},\alg{S},F^\un)$.  Then, by \cite[Lemma 3.4.4]{Kaletha:2019}, there exists an elliptic maximal $\res$-torus $\mathcal{S}$ of $\mathcal{G}_y$ such that for every unramified extension $F'$ of $F$, the image of $S(F')_0$ in $G(F')_{y,0:0^+}$ is equal to $\mathcal{S}(\res')$. Furthermore, every elliptic maximal $\res$-torus of $\mathcal{G}_y$ arises this way. By definition of regularity, the character $\theta|_{S_0}$ factors through to a character $\overline{\theta}$ of $\mathcal{S}(\res)$ such that its stabilizer in $N_\alg{G}(\alg{S})(F)/S$ is trivial. This character $\overline{\theta}$, also referred to as regular, is then in general position \cite[Fact 3.4.18]{Kaletha:2019}, meaning that $\pm R_{\mathcal{S},\overline{\theta}}$ is a Deligne-Lusztig cuspidal representation of $\mathcal{G}_y(\res) = G_{y,0:0^+}$ \cite[Theorem 8.3]{DL:1976}. Note that the sign $\pm$ refers to $(-1)^{r_{\res}(\mathcal{G}_y)-r_{\res}(\mathcal{S})}$, where $r_{\res}(\mathcal{G}_y)$ and $r_{\res}(\mathcal{S})$ denote the $\res$-split ranks of $\mathcal{G}_y$ and $\mathcal{S}$, respectively.

Let $\upkappa_{(\mathcal{S},\overline{\theta})}$ denote the pullback of $\pm R_{\mathcal{S},\overline{\theta}}$ to $G_{y,0}$. One can extend $\upkappa_{(\mathcal{S},\overline{\theta})}$ to an irreducible representation (acting on the same vector space) $\upkappa_{(\mathcal{S},\theta)}$ of $SG_{y,0}$, which is a normal subgroup of $G_{[y]}$. An explicit extension is provided in \cite[Section 3.4.4]{Kaletha:2019}, but for our purposes it suffices for us to know such an extension exists. The regularity of $\theta$ implies that $I_{G_{[y]}}(\upkappa_{(\alg{S},\theta)}) = SG_{y,0}$ \cite[Lemma 3.4.20]{Kaletha:2019} so that $\rho\defeq \Ind_{SG_{y,0}}^{G_{[y]}}\upkappa_{(\alg{S},\theta)}$ is irreducible, and therefore ${\pi_{(\alg{S},\theta)}\defeq \Ind_{G_{[y]}}^G\rho}$ is supercuspidal of depth zero as per \cite[Theorem 6.8]{MP:1996} and is regular by construction.

Because this construction exhausts all regular depth-zero supercuspidal representations \cite[Proposition 3.4.27]{Kaletha:2019}, when $\rho$ induces to such a representation, we can assume without loss of generality that $\rho = \Ind_{SG_{y,0}}^{G_{[y]}}\upkappa_{(\alg{S},\theta)}$ for some maximally unramified elliptic maximal $F$-torus $\alg{S}$ of $\alg{G}$ and regular depth-zero character $\theta$ of $S$.

\begin{figure}[!htbp]
\begin{tikzpicture}
  \matrix (m) [matrix of math nodes,row sep=3em,column sep=0.75em,minimum width=2em]
  {
     {} &{} &{} &{} &{} &\text{$\left( \rho, G_{[y]}\right)$} &{}   \\
     \text{$\left( \pm R_{\mathcal{S},\overline{\theta}}, \mathcal{G}_y(\res) \right)$} 
     &{} &\text{$\left( \upkappa_{(\mathcal{S},\overline{\theta})}, G_{y,0} \right)$} 
     &{} &\text{$\left( \upkappa_{(\alg{S},\theta)}, SG_{y,0} \right)$} 
     &{} &\text{$\left( \pi_{(\alg{S},\theta)}, G \right)$} \\};
  \path[-stealth]
    (m-2-1) edge node [above] {pullback} (m-2-3)
    (m-2-3)  edge node [above] {extend} (m-2-5)
    (m-2-5) edge node [above] {induce} (m-2-7)
    (m-2-5) edge node [left] {induce} (m-1-6)
    (m-1-6) edge node [right] {induce} (m-2-7);

\end{tikzpicture}

\caption[Summary of Kaletha's construction.]{Summary of Kaletha's construction for regular depth-zero supercuspidal representations of $G$.}
\label{fig:SummaryKaletha}
\end{figure}


Now, we proceed in steps to restrict $\rho$ to $H_{[y]}$.
We let $\mathcal{H}_y$ denote the reductive quotient of $\alg{H}$ at $y$. Recall that Lemma~\ref{lem:HxGx} allows us to view $\mathcal{H}_y$ as a subgroup of $\mathcal{G}_y$ that contains $[\mathcal{G}_y,\mathcal{G}_y]$.

\begin{lemma}\label{lem:equalityTori}
Let $\alg{S}$ be a maximally unramified elliptic maximal $F$-torus of $\alg{G}$ with associated vertex $[y]$ in $\mathcal{B}^\mathrm{red}(\alg{G},F)$, and let $\mathcal{S}$ be its corresponding elliptic maximal $\res$-torus of $\mathcal{G}_y$. Set $\alg{S}_\alg{H} = \alg{S}\cap\alg{H}$. Then $\alg{S}_{\alg{H}}$ is a maximally unramified elliptic maximal $F$-torus of $\alg{H}$ with associated vertex $[y]$ in $\mathcal{B}^{\mathrm{red}}(\alg{H},F)$. Furthermore, letting $\mathcal{S}_\mathcal{H}$ denote its corresponding elliptic maximal $\res$-torus of $\mathcal{H}_y$, we have that $\mathcal{S}_\mathcal{H} = \mathcal{S}\cap \mathcal{H}_y$.
\end{lemma}

\begin{proof}
That $\alg{S}_\alg{H}$ is a maximally unramified elliptic maximal $F$-torus is a consequence of Theorem~\ref{th:intersectTorus}. Furthermore, as $\alg{G}$ and $\alg{H}$ have the same derived subgroup, ${\mathcal{B}^{\mathrm{red}}(\alg{G},F) = \mathcal{B}^{\mathrm{red}}(\alg{H},F)}$ and $\mathcal{A}^{\mathrm{red}}(\alg{G},\alg{S},F^{\un}) = \mathcal{A}^{\mathrm{red}}(\alg{H},\alg{S}_\alg{H},F^{\un})$. It follows that $[y]$ is the unique $\Gal(F^{\un}/F)$-fixed point of $ \mathcal{A}^{\mathrm{red}}(\alg{H},\alg{S}_\alg{H},F^{\un})$ and is therefore the vertex associated to $\alg{S}_\alg{H}$.

 By definition, we have that $\mathcal{S}(\resun) = \quo{\left(\alg{S}(F^\un)_0\alg{G}(F^\un)_{y,0^+} \right)}{\alg{G}(F^\un)_{y,0^+}}$ and ${\mathcal{S}_\mathcal{H}(\resun) = \quo{\left(\alg{S}_\alg{H}(F^\un)_0\alg{H}(F^\un)_{y,0^+} \right)}{\alg{H}(F^\un)_{y,0^+}}}$. Let $\iso: \mathcal{H}_y(\resun)\rightarrow \quo{\alg{H}(F^\un)_{y,0}\alg{G}(F^\un)_{y,0^+}}{\alg{G}(F^\un)_{y,0^+}}$ be the isomorphism given by the second isomorphism theorem. When we write $(\mathcal{S}\cap\mathcal{H}_y)(\resun)$, we actually mean
 $$\iso^{-1}\left( \mathcal{S}(\resun) \cap \quo{\alg{H}(F^\un)_{y,0}\alg{G}(F^\un)_{y,0^+}}{\alg{G}(F^\un)_{y,0^+}} \right),$$ which is an elliptic maximal $\res$-torus of $\mathcal{H}_y(\resun)$. Since $\alg{S}_\alg{H}(F^\un)_0$ is contained in both $\alg{S}(F^\un)_0$ and $\alg{H}(F^\un)_{y,0}$,
$\mathcal{S}_\mathcal{H}(\resun) \subset (\mathcal{S}\cap\mathcal{H}_y)(\resun)$. The maximality of both tori implies that the inclusion is an equality.
\end{proof}

\begin{proposition}\label{prop:kappas}
Let $\alg{S}$ be a maximally unramified elliptic maximal $F$-torus of $\alg{G}$ associated to the vertex $[y]$ and let $\mathcal{S}$ be its corresponding elliptic maximal $\res$-torus of $\mathcal{G}_y$. Denote by $\mathcal{S}_{\mathcal{H}}$ the elliptic maximal $\res$-torus of $\mathcal{H}_y$ corresponding to the maximally unramified elliptic maximal $F$-torus $\alg{S}_\alg{H} = \alg{S}\cap\alg{H}$ of $\alg{H}$. Let $\overline{\theta}$ be a regular character of $\mathcal{S}(\res)$ and assume that $\overline{\theta}_{\mathcal{H}} = \overline{\theta}|_{\mathcal{S}_\mathcal{H} (\res)}$ is also regular. Then $$\upkappa_{\left(\mathcal{S},\overline{\theta}\right)}|_{H_{y,0}} = \upkappa_{\left(\mathcal{S}_\mathcal{H},\overline{\theta}_{\mathcal{H}}\right)},$$
where $\upkappa_{\left(\mathcal{S},\overline{\theta}\right)}$ and $\upkappa_{\left(\mathcal{S}_\mathcal{H},\overline{\theta}_{\mathcal{H}}\right)}$ denote the pullbacks of $\pm R_{\mathcal{S},\overline{\theta}}$ and $\pm R_{\mathcal{S}_\mathcal{H},\overline{\theta}_{\mathcal{H}}}$ to $G_{y,0}$ and $H_{y,0}$, respectively. That is, we have the diagram illustrated in Figure~\ref{fig:diagramCuspidalRestrict}.
\end{proposition}

\begin{figure}[!htbp]
\center
\begin{tikzpicture}
  \matrix (m) [matrix of math nodes,row sep=3em,column sep=4em,minimum width=2em]
  {
     \text{$\left( \pm R_{\mathcal{S},\overline{\theta}}, \mathcal{G}_y(\res) \right)$} 
     &\text{$\left( \pm R_{\mathcal{S}_\mathcal{H},\overline{\theta}_{\mathcal{H}}}, \mathcal{H}_y(\res) \right)$}  \\
     \text{$\left( \upkappa_{(\mathcal{S},\overline{\theta})}, G_{y,0} \right)$} 
     &\text{$\left( \upkappa_{(\mathcal{S}_\mathcal{H},\overline{\theta}_{\mathcal{H}})}, H_{y,0} \right)$} \\};
  \path[-stealth]
    (m-1-1)  edge node [right] {pullback} (m-2-1)
    (m-2-1) edge node [above] {restrict} (m-2-2)
    (m-1-2) edge node [right] {pullback} (m-2-2);

\end{tikzpicture}
\caption[Diagram summarizing the restriction of $\upkappa_{(\mathcal{S},\overline{\theta})}$ to $H_{y,0}$.]{Diagram summarizing the restriction of $\upkappa_{(\mathcal{S},\overline{\theta})}$ to $H_{y,0}$ as per Proposition~\ref{prop:kappas}.}
\label{fig:diagramCuspidalRestrict}
\end{figure}

\begin{proof}
For all $h\in H_{y,0}$, we have 
$$\upkappa_{(\mathcal{S}, \overline{\theta})}(h) = \pm R_{\mathcal{S}, \overline{\theta}}(hG_{y,0^+}) = \pm R_{\mathcal{S}, \overline{\theta}}|_{{\mathcal{H}}_y(\res)}(hG_{y,0^+}).$$ Since $[\mathcal{G}_y,\mathcal{G}_y] \subset {\mathcal{H}}_y$ (Lemma \ref{lem:HxGx}), it follows from Theorem \ref{th:restrictRTtheta} that $R_{\mathcal{S},\overline{\theta}}|_{{\mathcal{H}}_y(\res)} = R_{\mathcal{S}\cap {\mathcal{H}}_y, \overline{\theta}_{{\mathcal{H}}_y}}$, where $\overline{\theta}_{{\mathcal{H}}_y}$ denotes the restriction of $\overline{\theta}$ to $\left(\mathcal{S}\cap {{\mathcal{H}}_y} \right)(\res)$. Furthermore, one can show using \cite[Proposition 4.27]{BorelTits:1965} that $r_{\res}(\mathcal{G}_y) - r_{\res}(\mathcal{S}) = r_{\res}(\mathcal{H}_y) = r_{\res}(\mathcal{S}\cap\mathcal{H}_y)$ so that ${\pm R_{\mathcal{S},\overline{\theta}}|_{{\mathcal{H}}_y(\res)} = \pm R_{\mathcal{S}\cap {\mathcal{H}}_y, \overline{\theta}_{{\mathcal{H}}_y}}}$. Using Proposition~\ref{lem:equalityTori} and Theorem~\ref{th:iso}, we conclude that
$$\upkappa_{\left(\mathcal{S},\overline{\theta}\right)}(h) = \pm R_{\mathcal{S}_\mathcal{H},\overline{\theta}_{\mathcal{H}}}(hH_{y,0^+}) = \upkappa_{\left( \mathcal{S}_\mathcal{H},\overline{\theta}_{\mathcal{H}} \right)}(h). $$
\end{proof}

\begin{remark}\label{rem:regularRestriction}
There is no guarantee for $\overline{\theta}_{\mathcal{H}}$ to be regular when $\overline{\theta}$ is regular. However, under the hypothesis that $\overline{\theta}_{\mathcal{H}}$ is also regular, we know that $\pm R_{\mathcal{S}_\mathcal{H},\overline{\theta}_{\mathcal{H}}}$ is cuspidal and that the normalizer in $H_{[y]}$ of an extension to $S_HH_{y,0}$ of $\upkappa_{\left( \mathcal{S}_\mathcal{H},\overline{\theta}_{\mathcal{H}} \right)}$ is equal to $S_HH_{y,0}$.
\end{remark}

We note the following relationship between the groups $SG_{y,0}$ and $S_HH_{y,0}$.

\begin{lemma}\label{lem:SGSH}
Let $\alg{S}$ be an elliptic maximally unramified $F$-torus of $\alg{G}$ with associated vertex $[y]$ and let $\alg{H}$ be a closed connected $F$-subgroup of $\alg{G}$ that contains $\alg{G}_{\der}$. Then ${SG_{y,0}\cap H = S_HH_{y,0}}$, where $S_H = S\cap H$.
\end{lemma}

\begin{proof}
It is clear that $S_HH_{y,0} \subset SG_{y,0}\cap H$.

For the converse, take $h\in SG_{y,0}\cap H$. Then $h=sg$ for some $s\in S, g\in G_{y,0}$. We claim that $G_{y,0} = S_0H_{y,0}G_{y,0^+}$. Indeed, let $\mathcal{S}$ be the elliptic maximal $\res$-torus of $\mathcal{G}_y$ associated to $\alg{S}$. Because ${\mathcal{H}}_y$ contains $[\mathcal{G}_y,\mathcal{G}_y]$ (Lemma~\ref{lem:HxGx}), we have $\mathcal{G}_y = \mathcal{S} {\mathcal{H}}_y$. As a consequence of the Lang-Steinberg Theorem, it follows that 
$\mathcal{G}_y(\res) = \mathcal{S}(\res){\mathcal{H}}_y(\res),$ or equivalently $$\quo{G_{y,0}}{G_{y,0^+}} = \quo{S_0G_{y,0^+}}{G_{y,0^+}} \cdot \quo{H_{y,0}G_{y,0^+}}{G_{y,0^+}}.$$ Therefore, $G_{y,0} = S_0H_{y,0}G_{y,0^+}$ as claimed.

This means that we can rewrite $g = s'h'g'$ for some ${s'\in S_0}$, ${h'\in H_{y,0}}$, ${g'\in G_{y,0^+}}$ so that $h = ss'h'g'$. Furthermore, \cite[Lemma B.7.2]{DeBackerReeder:2009} allows us to write ${G_{y,0^+} = Z\overline{G}_{y,0^+}}$ where $Z=Z(\alg{G})^\circ(F)\subset S$ and $\overline{G} = \alg{G}_{\der}(F)\subset H$. This means that we can write $g' = z\overline{g}$ for some $z\in Z\subset S, \overline{g}\in\overline{G}_{y,0^+}\subset H_{y,0^+}$, meaning that
$$h = ss'h'z\overline{g} = ss'zh'\overline{g}.$$ But $ss'z = h(h'\overline{g})^{-1}\in S\cap H = S_H$ and $h'\overline{g}\in H_{y,0}$. Thus $h\in S_HH_{y,0}$ which concludes the proof.
\end{proof}

What we have pointed out in Remark~\ref{rem:regularRestriction}  motivates us to impose the additional condition that $\overline{\theta}_{\mathcal{H}}$ be regular in the following theorem. Note that this character $\overline{\theta}_{\mathcal{H}}$ canbe obtained by factoring the character $\theta|_{S_H}$ of $S_H$.

\begin{proposition}\label{prop:component}
Let $\alg{S}$ be an elliptic maximally unramified $F$-torus of $\alg{G}$ with associated vertex $[y]$. Let $\theta$ be a regular depth-zero character of $S$ and set ${\rho = \Ind_{SG_{y,0}}^{G_{[y]}}\upkappa_{\left(\alg{S},\theta \right)}}$. Let $\alg{H}$ be a closed connected $F$-subgroup of $\alg{G}$ that contains $\alg{G}_{\der}$ and set $\alg{S}_\alg{H} = \alg{S}\cap\alg{H}$. Assume that $\theta|_{S_H}$ is also regular. Then ${\rho' = \Ind^{H_{[y]}}_{S_HH_{y,0}}\left( \upkappa_{(\alg{S},\theta)}|_{S_HH_{y,0}} \right)}$ is a component of $\rho_H = \rho|_{H_{[y]}}$. Furthermore, the multiplicity of $\rho'$ in $\rho_H$ is equal to $m=[I_{G_{[y]}}(\rho'):H_{[y]}SG_{y,0}]$.
\end{proposition}

\begin{proof}
The Mackey Decomposition allows us to compute the restriction of $\rho$ to $H_{[y]}$ as follows:
\begin{align*}
\rho|_{H_{[y]}} &= \Res^{G_{[y]}}_{H_{[y]}}\Ind_{SG_{y,0}}^{G_{[y]}}\upkappa_{(\alg{S},\theta)} \\
&= \underset{c\in C}{\oplus} \Ind_{H_{[y]}\cap {^c(SG_{y,0})}}^{H_{[y]}}\Res^{^c(SG_{y,0})}_{H_{[y]}\cap {^c(SG_{y,0})}}{^c\upkappa_{(\alg{S},\theta)}},
\end{align*}
where $C$ is a set of coset representatives of $\Bigslant{H_{[y]}}{G_{[y]}}{SG_{y,0}}$. In particular, because $SG_{y,0}\cap H_{[y]} = SG_{y,0}\cap H = S_HH_{y,0}$ (Lemma~\ref{lem:SGSH}), we see that $\Ind_{S_HH_{y,0}}^{H_{[y]}}\left(\upkappa_{(\alg{S},\theta)}|_{S_HH_{y,0}}\right)$ is a subrepresentation of $\rho_H$. By Proposition~\ref{prop:kappas}, we see that $\upkappa_{(\alg{S},\theta)}|_{S_HH_{y,0}}$ is an extension to $S_HH_{y,0}$ of $\upkappa_{(\mathcal{S}_\mathcal{H},\overline{\theta}_{\mathcal{H}})}$. By hypothesis, we have assumed that $\theta|_{S_H}$ (and therefore $\overline{\theta}_{\mathcal{H}}$) is regular, meaning that the normalizer in $H_{[y]}$ of $\upkappa_{(\alg{S},\theta)}|_{S_HH_{y,0}}$ is equal to $S_HH_{y,0}$ as discussed in Remark~\ref{rem:regularRestriction}. Therefore, $\rho' = \Ind_{S_HH_{y,0}}^{H_{[y]}}\left(\upkappa_{(\alg{S},\theta)}|_{S_HH_{y,0}} \right)$ is irreducible, meaning that $\rho'$ is a component of $\rho_H$.

Now, for the multiplicity, we know from Proposition~\ref{prop:multiplicityrhoH} that the multiplicity of $\rho'$ in $\rho_H$ is given by
$$m = \frac{\dim(\rho)}{[G_{[y]}:I_{G_{[y]}}(\rho')]\dim(\rho')},$$ where $I_{G_{[y]}}(\rho')$ is the normalizer in $G_{[y]}$ of $\rho'$. We observe that ${\dim(\rho) = [G_{[y]}:SG_{y,0}]\dim(\upkappa_{(\alg{S},\theta)})}$ and $\dim(\rho') = [H_{[y]}:S_HH_{y,0}]\dim(\upkappa_{(\alg{S},\theta)})$, which implies that
$$m = \frac{[G_{[y]}:SG_{y,0}]}{[G_{[y]}:I_{G_{[y]}}(\rho')][H_{[y]}:S_HH_{y,0}]}.$$ But $[H_{[y]}:S_HH_{y,0}] = [H_{[y]}SG_{y,0}:SG_{y,0}]$ as $S_HH_{y,0} = SG_{y,0}\cap H_{[y]}$ (Lemma~\ref{lem:SGSH}), implying that $\quo{H_{[y]}}{S_HH_{y,0}}\simeq\quo{H_{[y]}SG_{y,0}}{SG_{y,0}}$ by the second isomorphism theorem. Therefore,
$$m=\frac{[G_{[y]}:SG_{y,0}]}{[G_{[y]}:I_{G_{[y]}}(\rho')][H_{[y]}SG_{y,0}:SG_{y,0}]},$$ which simplifies to $m= [I_{G_{[y]}}(\rho'):H_{[y]}SG_{y,0}]$ as $H_{[y]}SG_{y,0}\subset I_{G_{[y]}}(\rho') \subset G_{[y]}$.
\end{proof}

Without the hypothesis that $\theta|_{S_H}$ is also regular in the previous proposition, one would need to decompose $\Ind_{S_HH_{y,0}}^{H_{[y]}}\left( \upkappa_{(\alg{S},\theta)}|_{S_HH_{y,0}} \right)$ in order to find a component $\rho'$ of $\rho_H$. Furthermore, this decomposition could potentially lead to additional multiplicity in $\rho_H$.

If $\rho$ satisfies the hypotheses of Proposition~\ref{prop:component} and that $I_{G_{[y]}}(\rho') = H_{[y]}SG_{y,0}$, then the restriction $\rho_H$ will be multiplicity free. For this to happen, we need to impose a slightly stronger condition on the character $\theta$. Following the language of \cite{AM:2019}, we have the following definition.\footnote{The version of the pre-print \cite{AM:2019} cited here contains a typo in its Definition 5.2, which was confirmed by one of the authors through a private correspondence.}

\begin{definition}
Let $\theta'$ be a character of $S_H$. We say that $\theta'$ is \emph{regular in $G$} if $I_{N_{\alg{G}}(\alg{S})(F)}(\theta'|_{(S_H)_0}) = S$.
\end{definition}

It is easy to verify that if $\theta'$ is a character of $S_H$ which is regular in $G$, then it is a regular in the sense of \cite[Definition 3.4.16]{Kaletha:2019}. 

\begin{corollary}\label{cor:rhoMultFree}
Let $\alg{S},\alg{H},\alg{S}_\alg{H}, \theta, \rho$ and $\rho'$ be as in Proposition~\ref{prop:component}. Assume furthermore that $\theta|_{S_H}$ is regular in $G$. Then the restriction $\rho_H =\rho|_{H_{[y]}}$ is multiplicity free.
\end{corollary}

\begin{proof}
By the previous proposition, it suffices to show that $I_{G_{[y]}}(\rho') \subset H_{[y]}SG_{y,0}$ to obtain equality between the two sets.

Let $g\in I_{G_{[y]}}(\rho')$. Then $^g\Ind_{S_HH_{y,0}}^{H_{[y]}}\left( \upkappa_{(\alg{S},\theta)}|_{S_HH_{y,0}} \right)\simeq \Ind_{S_HH_{y,0}}^{H_{[y]}}\left( \upkappa_{(\alg{S},\theta)}|_{S_HH_{y,0}} \right)$. From the Mackey Decomposition, it follows that ${^g\left(\upkappa_{(\alg{S},\theta)}|_{S_HH_{y,0}} \right) \simeq {^h\left( \upkappa_{(\alg{S},\theta)}|_{S_HH_{y,0}}\right)}}$ for some $h\in H_{[y]}$, or equivalently $^{h^{-1}g}\left(\upkappa_{(\alg{S},\theta)}|_{S_HH_{y,0}}\right) \simeq \left(\upkappa_{(\alg{S},\theta)}|_{S_HH_{y,0}}\right)$. Using Kaletha's argument in \cite[Lemma 3.4.20]{Kaletha:2019}, there exists $g'\in H_{y,0}$ such that $g'h^{-1}g\in N_{\alg{G}}(\alg{S}_{\alg{H}})(F) = N_{\alg{G}}(\alg{S})(F)$ and ${^{g'h^{-1}g}\left(\theta|_{(S_H)_0}\right) = \theta|_{(S_H)_0}}$. Equivalently, $g'h^{-1}g\in I_{N_{\alg{G}}(\alg{S})(F)}\left( \theta|_{(S_H)_0} \right)$. Therefore, $g'h^{-1}g\in S$ and $g\in H_{[y]}SG_{y,0}$.
\end{proof}

Using Theorem~\ref{th:multiplicity}, we restate the previous result in terms of the regular supercuspidal representation. The statement we provide below was also proved by Adler and Mishra in a recent preprint \cite[Theorem 5.3]{AM:2019}. While their proof is very concise, our method provides a more explicit description of the restriction.

\begin{corollary}\label{cor:multiplicityRegular}
Let $\pi_{(\alg{S},\theta)}$ be a regular depth-zero supercuspidal representation of $G$, where $\alg{S}$ is an elliptic maximally unramified $F$-torus and $\theta$ is a regular depth-zero character of $S$. Let $\alg{H}$ be a closed connected $F$-subgroup of $\alg{G}$ that contains $\alg{G}_{\der}$ and set $\alg{S}_\alg{H} = \alg{S}\cap\alg{H}$. Assume furthermore that $\theta|_{S_H}$ is regular in $G$. Then $\pi_{(\alg{S},\theta)}|_H$ is multiplicity free.
\end{corollary}

Finally, using Corollary~\ref{cor:multiplicityDepthZero}, this allows us to write a statement for a multiplicity free restriction of an arbitrary irreducible supercuspidal representation. 

\begin{corollary}\label{cor:multiplicityFreeRes}
Let $\Psi = (\vec{\alg{G}},y,\vec{r},\rho,\vec{\phi})$ be a $G$-datum. Assume that ${\pi^{-1} = \Ind_{G^0_{[y]}}^{G^0}\rho}$ is a regular depth-zero supercuspidal representation that satisfies the hypotheses of Corollary~\ref{cor:multiplicityRegular} when setting $\alg{G}=\alg{G}^0$ and $\alg{H}=\alg{H}^0$. Then $\pi_G(\Psi)|_H$ is multiplicity free.
\end{corollary}

\section{A Restriction of Types for Non-Supercuspidal Representations}\label{sec:types}

So far, we have focused on restricting the irreducible supercuspidal representations $G$, but this is only a fraction of the irreducible representations of $G$. The task of studying the irreducible representations of $G$ which are not supercuspidal is a whole different problem. Unlike the supercuspidal representations, irreducible non-supercuspidal representations are not all obtained by compact induction from an open compact-mod-center subgroup. Rather, one can turn to the theory of types to obtain information, which applies to all representations of $G$.

The theory of types was introduced by Bushnell and Kutzko in \cite{BK:1998}. This theory allows to describe representations of $G$ in terms of representations of certain compact open subgroups. To define what a type is, one needs the Bernstein decomposition \cite{Bernstein:1984} for the category $\mathcal{R}(G)$ of smooth complex representations of $G$. This decomposition allows to write $\mathcal{R}(G)$ as a direct product of subcategories over some index set $\mathfrak{B}$
$$\mathcal{R}(G) = \prod\limits_{\mathfrak{s}\in\mathfrak{B}} \mathcal{R}^\mathfrak{s}(G).$$ Elements of $\mathfrak{B}$ correspond to inertial equivalence classes of pairs $(M,\sigma)$, where $M$ is a Levi subgroup (of a parabolic subgroup) of $G$ and $\sigma$ is an irreducible supercuspidal representation of $M$. The subcategories $\mathcal{R}^\mathfrak{s}(G),\mathfrak{s}\in\mathfrak{B},$ are called Berstein blocks. Types give an explicit description of each Berstein block as per the following definition.

\begin{definition}
Let $\mathfrak{s}\in\mathfrak{B}$, $J$ be a compact open subgroup of $G$ and $\lambda$ be a smooth irreducible representation of $J$. We say that $(J,\lambda)$ is an $\mathfrak{s}$-type if for every smooth irreducible representation $\pi$ of $G$, we have that $\pi \in \mathcal{R}^\mathfrak{s}(G)$ if and only if $\pi|_J$ contains $\lambda$.

We say that an irreducible smooth representation $\pi$ of $G$ contains a type if there exists an $\mathfrak{s}$-type $(J,\lambda)$ for some $\mathfrak{s}\in\mathfrak{B}$ such that $\pi|_J$ contains $\lambda$.
\end{definition}

Thus, we can obtain a description of $\mathcal{R}(G)$ entirely in terms of representations of compact open groups via the construction of an $\mathfrak{s}$-type for each $\mathfrak{s}\in\mathfrak{B}$.

Kim and Yu provided a construction for types, which we shall refer to as \emph{Kim-Yu types}, that is completely analogous to the J.K.~Yu construction in \cite{KimYu}. Furthermore, they showed that every smooth irreducible representation of $G$ contains such a type when $p$ is large enough \cite[Theorem 9.1]{KimYu}. Fintzen refined this exhaustion condition to when $p$ does not divide $|W|$ \cite[Theorem 7.12]{Fintzen:2018}. This exhaustion justifies the importance of studying Kim-Yu types.

To construct a type, one starts from a 5-tuple, which we will refer to as a \emph{type-datum}. This type-datum is a modified version of a $G$-datum and is defined as follows \cite[Section 7.2]{KimYu}.

\begin{definition}
A sequence  $\Sigma = ((\vec{\alg{G}},\alg{M}^0),(y,\{I\}),\vec{r},(M^0_y,\tilde{\rho}),\vec{\phi})$ is a type-datum for $G$ if and only if:
\begin{enumerate}
\item[TD1)] $\vec{\alg{G}}$ is a tamely ramified twisted Levi sequence $\vec{\alg{G}} = (\alg{G}^0,\dots,\alg{G}^d)$ in $\alg{G}$, and $\alg{M}^0$ is a Levi subgroup of $\alg{G}^0$;
\item[TD2)] $y$ is a point in $\mathcal{B}(\alg{M}^0,F)$ and $\{I\}$ is a commutative diagram of embeddings of buildings which is $\vec{s}$-generic relative to $y$ in the sense of \cite[Section 3]{KimYu}, where ${\vec{s} = (0,r_0/2,\dots,r_d/2)}$;
\item[TD3)] $\vec{r} = (r_0,\dots,r_d)$ is a sequence of real numbers satisfying $0 <r_0 <\cdots < r_{d-1} \leq r_d$ if $d>0, 0\leq r_0$ if $d=0$;
\item[TD4)] $M^0_{y,0}$ is a maximal parahoric subgroup of $M^0$, and $\tilde{\rho}$ is an irreducible smooth representation of $M^0_y$ such that $\tilde{\rho}|_{M^0_{y,0}}$ contains a cuspidal representation of $M^0_{y,0:0^+}$;
\item[TD5)] $\vec{\phi} = (\rep{0},\dots,\rep{d})$ is a sequence of quasicharacters, where $\rep{i}$ is a $G^{i+1}$-generic character of $G^i$ of depth $r_i$ relative to $x$ for all $x\in \mathcal{B}(\alg{G}^i,F)$, $0\leq i\leq d-1$. If $r_{d-1}<r_d$, we assume $\rep{d}$ is of depth $r_d$, otherwise $\rep{d}=1$.
\end{enumerate}
\end{definition}





It is clear from the definition of type-datum that we can adapt the statement of Theorem~\ref{th:GderData} to obtain a restriction of type-datum as follows.

\begin{theorem}\label{th:typeData}
Let $\Sigma = ((\vec{\alg{G}},\alg{M}^0),(y,\{I\}),\vec{r},(M^0_y,\tilde{\rho}),\vec{\phi})$ be a type-datum for $G$.
Let ${\alg{M}^0_\alg{H} = \alg{M}^0\cap\alg{H}}$, $\alg{H}^i = \alg{G}^i\cap\alg{H}$ and $\repH{i} = \rep{i} |_{H^i}$ for all $0\leq i\leq d$. Let $\tilde{r}$ denote the depth of $\repH{d}$.
\begin{enumerate}
\item[1)] If $\tilde{r} > r_{d-1}$, set $\vec{\phi_H}=(\repH{0},\dots,\repH{d-1},\repH{d})$ and $\vec{\tilde{r}} = (r_0,\dots,r_{d-1},\tilde{r})$.
\item[2)] If $\tilde{r}\leq r_{d-1}$, set $\vec{\phi_H} = (\repH{0},\cdots,\repH{d-1}\repH{d},1)$ and $\vec{\tilde{r}} = (r_0,\dots, r_{d-1}, r_{d-1})$.
\end{enumerate}
Then, for all irreducible subrepresentation $\tilde{\rho}_\ell$ of $\tilde{\rho} |_{(M^0_H)_y}$,  ${\Sigma_\ell = ((\vec{\alg{H}},\alg{M}^0_\alg{H}),(y,\{I\}),\vec{\tilde{r}},((M^0_H)_y,\tilde{\rho}_\ell),\vec{\phi_H})}$ is a type-datum for $H$, where $\vec{\alg{H}} = (\alg{H}^0,\dots,\alg{H}^d)$.
\end{theorem}

\begin{proof}
The same arguments that we used for Theorem~\ref{th:GderData} can be used to prove axioms TD1) and TD3)-TD5) for a type-datum for $H$.

For TD2), we can assume without loss of generality that $y$ belongs to $\mathcal{B}(\alg{M}^0_H,F)$ as per Lemma~\ref{lem:[x]} and Remark~\ref{rem:WLOGy}. Furthermore, the natural embedding $\mathcal{B}(\alg{H}^i,F)\hookrightarrow\mathcal{B}(\alg{G}^i,F)$ from Remark~\ref{rem:WLOGy} means that $\{I\}$ induces a commutative diagram of embeddings on $\vec{\alg{H}}$ and $\alg{M}^0_\alg{H}$ which is $\vec{s}$-generic in the sense of \cite[Section 3]{KimYu}.
\end{proof}

Given a type-datum $\Sigma$, let $J_{(\vec{\alg{G}},\alg{M}^0)} = M^0_yG^1_{y,s_0}\cdots G^d_{y,s_{d-1}}$, which is an open compact subgroup of $G$. A process completely analogous to the construction of $\kappa_G(\Psi)$, $\Psi$ a $G$-datum (Figure~\ref{fig:JKYuConstruction}), results in an irreducible representation $\lambda_G(\Sigma)$ of $J_{(\vec{\alg{G}},\alg{M}^0)}$. The pair $(J_{(\vec{\alg{G}},\alg{M}^0)},\lambda_G(\Sigma))$ is an $\mathfrak{s}$-type for some $\mathfrak{s}\in\mathfrak{B}$ \cite[Theorem 7.5]{KimYu}. As mentioned before, we call such a type a Kim-Yu type. 
Furthermore, we can adapt the proofs of Theorem~\ref{th:phiprime}, Proposition~\ref{prop:kappai} and Theorem~\ref{th:kappa}, which allowed us to compute $\kappa_G(\Psi)|_{\KH{d}}$, and obtain the following.

\begin{theorem}\label{th:typeRes}
Let $\Sigma = ((\vec{\alg{G}},\alg{M}^0),(y,\{I\}),\vec{r},(M^0_y,\rho),\vec{\phi})$ be a type-datum for $G$. Assume that the decomposition into components of $\rho|_{(M^0_H)_y}$ is given by ${\rho|_{(M^0_H)_y} = \underset{\ell\in L}{\oplus}\rho_\ell}$ for some index set $L$. For each $\ell\in L$, let $\Sigma_\ell$ be the type-datum for $H$ as per Theorem~\ref{th:typeData},  ${J_{(\vec{\alg{H}},\alg{M}^0_\alg{H})} = (M^0_H)_yH^1_{y,s_0}\cdots H^d_{y,s_{d-1}}}$ and $\lambda_H(\Sigma_\ell)$ be the irreducible representation of $J_{(\vec{\alg{H}},\alg{M}^0_\alg{H})}$ constructed from $\Sigma_\ell$. Then
$$\lambda_G(\Sigma)|_{J_{(\vec{\alg{H}},\alg{M}^0_\alg{H})}} = \underset{\ell\in L}{\oplus}\lambda_H(\Sigma_\ell).$$
\end{theorem}

Note that we can also  define an extension of types to go from a type-datum for $H$ to type-data for $G$ as per Theorem~\ref{th:backwards}.

\begin{remark}
Given a $G$-datum $\Psi = (\vec{\alg{G}},y,\vec{r},\rho,\vec{\phi})$, one can set $\Sigma = ((\vec{\alg{G}},\alg{G}^0),y,\vec{r},(G^0_y,\tilde{\rho}),\vec{\phi})$, where $\tilde{\rho}$ is a component of $\rho|_{G^0_y}$ to construct a Kim-Yu type $(J_{(\vec{\alg{G}},\alg{G}^0)},\lambda_G(\Sigma))$. Furthermore, $J_{(\vec{\alg{G}},\alg{G}^0)}$ is the maximal compact subgroup of $K^d$ and $\lambda_G(\Sigma)$ is an irreducible subrepresentation of $\kappa_G(\Psi)|_{J_{(\vec{\alg{G}},\alg{G}^0)}}$
\end{remark}

\begin{corollary}
Let $\pi$ be an irreducible representation of $G$ and let $(J_\Sigma,\lambda_G(\Sigma))$ be a Kim-Yu type contained in $\pi$. Let $\{\Sigma_\ell,\ell\in L\}$ be the type-data for $H$ obtained from $\Sigma$ as per Theorem~\ref{th:typeData}. Then, for all $\ell\in L$, $(J_{(\vec{\alg{H}},\alg{M}^0_\alg{H})},\lambda_H(\Sigma_\ell))$ is a Kim-Yu type which is contained in a component of $\pi|_H$.
\end{corollary}

In the case where $\pi = \pi_G(\Psi)$ is a supercuspidal representation of $G$, elements of $\Res(\Psi)$ did not exhaust the $H$-data for the components of $\pi|_H$ (Theorem~\ref{th:restrictSupercuspidal}). Therefore, for an arbitrary irreducible representation $\pi$ of $G$, we cannot expect the restriction of types defined above to yield a type for each component of $\pi|_H$. Furthermore, at this stage, we do not know whether we can expect a decomposition formula like that of Theorem~\ref{th:restrictSupercuspidal} for an irreducible representation $\pi$ which is not supercuspidal, but future work could investigate the direct implications of the restriction of Kim-Yu types on the decomposition of $\pi|_H$.

\appendix
\makeatletter
\renewcommand{\@seccntformat}[1]{}
\makeatother

\section{Appendix: Computations on Deligne-Lusztig Virtual Characters}\label{sec:appendix}

Let $\alg{G}$ denote a reductive group over an algebraically closed field of characteristic $p$. Let $\Fr:\alg{G}\rightarrow \alg{G}$ be a Frobenius map (as per \cite[Section 1.17]{Carter:1993}) and $\alg{G}^{\Fr} = \{g\in\alg{G}: \Fr(g)=g \}$. For our purposes, it suffices to know that the map $\Fr$ is a surjective homomorphism for which $\alg{G}^{\Fr}$ is finite. The group $\alg{G}^{\Fr}$ is called a \emph{finite group of Lie type} and Deligne and Lusztig defined a class of generalized (or virtual) characters for $\alg{G}^{\Fr}$ in \cite{DL:1976}. 

\begin{remark}\label{rem:finiteGroupLieType}
If $\alg{G}$ is a reductive group defined over the finite field $\res$, then $\alg{G}(\res)$ is a finite group of Lie type where the Frobenius map corresponds to taking the $\Gal(\resun/\res)$-points of $\alg{G}$. Furthermore, there exists a class of finite groups of Lie type which do not correspond to taking rational points of a reductive group, called the Suzuki and Ree groups \cite[Section 1.19]{Carter:1993}. For this reason, we phrase the results of this appendix in terms of the language of finite groups of Lie type even if we are only interested in $\res$-points of reductive groups in this paper. 
\end{remark}

Given a maximal torus $\alg{T}$ of $\alg{G}$ which is $\Fr$-stable and an irreducible complex character $\theta$ of $\alg{T}^{\Fr}$, we can construct a \emph{virtual character} $R_{\alg{T},\theta}$ which maps $\alg{G}^{\Fr}$ into $\mathbb{C}$.

The original description of $R_{\alg{T},\theta}$ provided in \cite{DL:1976} relies on $l$-adic cohomology groups with compact support of a variety $\tilde{X}$, where $l$ is a prime number different from $p$. These $l$-adic cohomology groups are denoted by $H^i_c(\tilde{X},\overline{\mathbb{Q}}_l)$, $i\geq 0$ and are described in \cite{DL:1976, Carter:1993} .

To define the variety $\tilde{X}$, we choose a Borel group $\alg{B}$ that contains $\alg{T}$. Note that this Borel subgroup need not be $\Fr$-stable. We have a Levi decomposition of $\alg{B}$, $\alg{B} = \alg{U}\alg{T}$, where $\alg{U}\subset\alg{G}_{\der}$ is some unipotent subgroup. Then, we define the variety $\tilde{X}$ as $\tilde{X}\defeq L_{\alg{G}}^{-1}(\alg{U})$, where $L_{\alg{G}}$ is the Lang map of $\alg{G}$ defined by
\begin{align*}
L_\alg{G} : \alg{G} &\rightarrow \alg{G} \\
g &\mapsto g^{-1}\Fr(g).
\end{align*}
Note that the Lang-Steinberg Theorem states that this map is surjective. 

The groups $\alg{G}^{\Fr}$ and $\alg{T}^{\Fr}$ act on $\tilde{X}$ by left multiplication and right multiplication, respectively \cite[Section 7.2]{Carter:1993}. Those actions then induce actions on the cohomology groups $H^i_c(\tilde{X},\overline{\mathbb{Q}}_l)$. The original formula for the virtual character $R_{\alg{T},\theta}$ provided in \cite[Section 1.20]{DL:1976} is based on the action of $\alg{G}^{\Fr}$ and is given as follows:
$$R_{\alg{T},\theta}(g) = \sum\limits_{i\geq 0} (-1)^i\mathrm{trace}(g,H^i_c(\tilde{X},\overline{\mathbb{Q}}_l)_\theta) \text{ for all } g\in \alg{G}^{\Fr},$$
where $H^i_c(\tilde{X},\overline{\mathbb{Q}}_l)_\theta$ is the $\alg{T}^{\Fr}$-submodule of $H^i_c(\tilde{X},\overline{\mathbb{Q}}_l)$ on which $\alg{T}^{\Fr}$ acts by the character $\theta$.

For a unipotent element $u\in\alg{G}^{\Fr}$, we set 
$$Q_{\alg{T}}^{\alg{G}}(u) = R_{\alg{T},1}(u).$$ The map $Q_{\alg{T}}^{\alg{G}}$ from the set of unipotent elements is called a \index{Green function} \emph{Green function}. The Green function $Q_{\alg{T}}^{\alg{G}}(u)$ is integer-valued \cite[Section 4]{DL:1976}. Furthermore, the formula of the virtual character can be expressed in terms of Green functions as follows \cite[Theorem 4.2]{DL:1976}. Given $g\in \alg{G}^{\Fr}$, we let $g=su=us$ be its Jordan decomposition. Then
\begin{align}\label{eq:green}
R_{\alg{T},\theta}(g) = \frac{1}{|(C_\alg{G}(s)^\circ)^{\Fr}|}\sum\limits_{\substack{x\in \alg{G}^{\Fr}\\ x^{-1}sx \in \alg{T}^{\Fr}}}\theta(x^{-1}sx)Q^{C_\alg{G}(s)^\circ}_{x\alg{T}x^{-1}}(u).
\end{align}

\begin{remark}
The formula for $R_{\alg{T},\theta}$ does not depend on the choice of Borel subgroup $\alg{B}$ \cite[Corollary 4.3]{DL:1976},\cite[Proposition 7.3.6]{Carter:1993}. 
\end{remark}

\subsection{Isomorphisms and Virtual Characters}

There is a second alternative formula for the virtual character that is worth noting. By what precedes above, for all $g\in \alg{G}^{\Fr}, t\in \alg{T}^{\Fr}$, we can view $(g,t)$ as an automorphism of $\tilde{X}$, defined by $(g,t)\tilde{x} = g\tilde{x}t$ for all $\tilde{x} \in \tilde{X}$. By \cite[Property 7.1.3]{Carter:1993}, this automorphism induces a non-singular map of $H^i_c(\tilde{X},\overline{\mathbb{Q}}_l)$ into itself. This action allows us to define what we call the \index{Lefschetz number} \emph{Lefschetz number} of $(g,t)$ on $\tilde{X}$,
$$\mathscr{L}_\alg{G}((g,t),\tilde{X}) = \sum\limits_{i\geq 0}(-1)^i\mathrm{trace}((g,t), H^i_c(\tilde{X},\overline{\mathbb{Q}}_l)).$$
The Lefschetz number is an integer that does not depend on the choice of the prime number $l$ ($l\neq p$) \cite[Property 7.1.4]{Carter:1993}.

Now, we have that the Frobenius map commutes with $(g,t)$ as maps of $\tilde{X}$.
Following \cite[Appendix (h)]{Carter:1993}, we then have $$\mathscr{L}_\alg{G}((g,t), \tilde{X}) = \lim\limits_{s\rightarrow\infty} \overline{R}_\alg{G}(s),$$ where $\overline{R}_\alg{G}(s)$ is the rational function with power series expansion $-\sum\limits_{n=1}^\infty \left| \tilde{X}^{\Fr^n\circ (g,t)^{-1}} \right| s^n$. Here $\tilde{X}^{\Fr^n\circ (g,t)^{-1}}$ denotes the elements of $\tilde{X}$ that are fixed by the map $\Fr^n\circ (g,t)^{-1}$. 

The Lefschetz number provides us with a second alternative formula for $R_{\alg{T},\theta}$, one that no longer depends on the $l$-adic cohomology subgroups by what precedes. For all $g\in \alg{G}^{\Fr}$, this formula is given by
\begin{align}\label{eq:Lef}
R_{\alg{T},\theta}(g) = \frac{1}{|\alg{T}^{\Fr}|}\sum\limits_{t\in\alg{T}^{\Fr}} \theta(t^{-1})\mathscr{L}_\alg{G}((g,t),\tilde{X}),
\end{align}
where $\mathscr{L}_\alg{G}((g,t),\tilde{X})$ is the Lefschetz number of $(g,t)$ on $\tilde{X}$ \cite[Proposition 7.2.3]{Carter:1993}.

Using this second alternative formula (\ref{eq:Lef}), we can establish an invariance under isomorphisms for the virtual characters as per the following theorem.

\begin{theorem}\label{th:iso}
Let $\alg{G}$ be a reductive group with Frobenius map $\Fr$ and $\alg{T}$ be a maximal $\Fr$-stable torus of $\alg{G}$. Let $f:\alg{G}\rightarrow \alg{G}'$ be an isomorphism of algebraic groups. Then, for all $g\in\alg{G}^{\Fr}$ we have
$R_{\alg{T},\theta}(g) = R_{\alg{T}',\theta\circ f^{-1}}(f(g)).$
\end{theorem}

\begin{proof}
First note that $\Fr_f \defeq f\circ \Fr \circ f^{-1}$ is a Frobenius map on $\alg{G}'$, $f(\alg{G}^{\Fr}) = {\alg{G}'}^{\Fr_f}$ and $f(\alg{T})$ is a maximal $\Fr_f$-stable maximal torus of $\alg{G}'$. Using (\ref{eq:Lef}), the formula for $R_{f(\alg{T}),\theta\circ f^{-1}}$ can be expressed in terms of the Lefschetz number on $\tilde{X}_f$, where $\tilde{X}_f = L_{\alg{G}'}^{-1}(f(\alg{U}))$ and the Lang map of $\alg{G}'$ is expressed in terms of $\Fr_f$.

One can easily verify from the definitions that $\tilde{X}_f = f(\tilde{X})$ and ${\tilde{X}_f^{\Fr_f^n\circ (f(g),f(t))^{-1}} = f(\tilde{X}^{\Fr^n\circ (g,t)^{-1}})}$ for all $g\in \alg{G}^{\Fr}, t\in \alg{T}^{\Fr}$. It follows that $|\tilde{X}_f^{\Fr_f^n\circ (f(g),f(t))^{-1}}| = |\tilde{X}^{\Fr^n\circ (g,t)^{-1}}|$ and therefore $\mathscr{L}_{\alg{G}}((g,t),\tilde{X}) = \mathscr{L}_{\alg{G}'}((f(g),f(t)),f(\tilde{X}))$ for all $g\in\alg{G}^{\Fr}, t\in \alg{T}^{\Fr}$ when using the power series expansion. The conclusion follows.

\end{proof}

\subsection{Restriction of Virtual Characters}

\begin{theorem}\label{th:restrictRTtheta}
Let $\alg{G}$ be a reductive group defined over an algebraically closed field of characteristic $p$ and $\Fr:\alg{G} \rightarrow \alg{G}$ be a Frobenius map. Let $\alg{T}$ be a maximal $\Fr$-stable torus of $\alg{G}$ and $\alg{H}$ be an $\Fr$-stable subgroup of $\alg{G}$ which is closed, connected and contains $\alg{G}_{\der}$. Then $R_{\alg{T},\theta}|_{\alg{H}^{\Fr}} = R_{\alg{T}_{\alg{H}},\theta_{\alg{H}}}$, where $\alg{T}_{\alg{H}} = \alg{T}\cap\alg{H}$ and $\theta_{\alg{H}} = \theta|_{\alg{T}_\alg{H}^{\Fr}}$.
\end{theorem}

\begin{proof}
Let $g = su$ be the Jordan decomposition of $g\in \alg{H}^{\Fr}$. Since the Jordan decomposition is unique, this must also correspond to the decomposition of $g$ in $\alg{G}^{\Fr}$. Since $\alg{G} = \alg{H}\alg{T}$, we have that $\quo{\alg{G}^{\Fr}}{\alg{H}^{\Fr}} \simeq \quo{\alg{T}^{\Fr}}{\alg{T}_\alg{H}^{\Fr}}$ (or $\alg{G}^{\Fr} = \alg{H}^{\Fr}\alg{T}^{\Fr}$) as a consequence of the Lang-Steinberg Theorem. Furthermore, since $s\in\alg{H}^{\Fr}$ and $\alg{H}^{\Fr}$ is normal in $\alg{G}^{\Fr}$, requiring $x^{-1}sx\in\alg{T}^{\Fr}$ is equivalent to requiring $x^{-1}sx\in\alg{T}_{\alg{H}}^{\Fr}$. Therefore, we can start modifying the formula from equation (\ref{eq:green}) for the restricted virtual character as follows:
\begin{align*}
R_{\alg{T,}\theta}|_{\alg{H}^{\Fr}}(g) &= \frac{1}{|(C_\alg{G}(s)^\circ)^{\Fr}|}\sum\limits_{ \substack{x\in \alg{G}^{\Fr} \\ x^{-1}sx\in \alg{T}^{\Fr}}}\theta(x^{-1}sx)Q^{C_\alg{G}(s)^\circ}_{x\alg{T}x^{-1}}(u) \\
&= \frac{1}{|(C_\alg{G}(s)^\circ)^{\Fr}|}\sum\limits_{\substack{x\in \alg{H}^{\Fr}\alg{T}^{\Fr} \\ x^{-1}sx\in \alg{T}_{\alg{H}}^{\Fr}}}\theta_\alg{H}(x^{-1}sx)Q^{C_\alg{G}(s)^\circ}_{x\alg{T}x^{-1}}(u) .
\end{align*}
Now, let $\{t_1,\dots, t_m\}$ be a set of coset representatives of $\alg{T}_{\alg{H}}^{\Fr} \setminus \alg{T}^{\Fr} \simeq \alg{H}^{\Fr} \setminus \alg{H}^{\Fr}\alg{T}^{\Fr}$. Then $\alg{H}^{\Fr}\alg{T}^{\Fr} = \alg{H}^{\Fr}t_1 \cup \cdots \cup \alg{H}^{\Fr}t_m$. Given $x\in \alg{H}^{\Fr}\alg{T}^{\Fr}$, we have $x= \bar{g}t_i$ for some $1\leq i\leq m$, $\bar{g}\in\alg{H}^{\Fr}$. Note that the right coset notation is necessary to have our elements of $\alg{H}^{\Fr}$ expressed in this particular way to ensure the following simplifications. Since $t_i \in \alg{T}^{\Fr}$, we have that $x\alg{T}x^{-1} = \bar{g}\alg{T}\bar{g}^{-1}$ and $\theta(x^{-1}sx) = \theta(t_i^{-1}\bar{g}^{-1}s\bar{g}t_i) = \theta(\bar{g}^{-1}s\bar{g}).$ This allows us to rewrite the restricted virtual character as follows:
\begin{align*}
R_{\alg{T,}\theta}|_{\alg{H}^{\Fr}}(g) &= \frac{1}{|(C_\alg{G}(s)^\circ)^{\Fr}|}\sum\limits_{\substack{\bar{g}\in\alg{H}^{\Fr}\\ 1\leq i\leq m\\ \bar{g}^{-1}s\bar{g}\in \alg{T}_{\alg{H}}^{\Fr}}}\theta_\alg{H}(\bar{g}^{-1}s\bar{g})Q^{C_\alg{G}(s)^\circ}_{\bar{g}\alg{T}\bar{g}^{-1}}(u)  \\
&= \frac{m}{|(C_\alg{G}(s)^\circ)^{\Fr}|}\sum\limits_{\substack{\bar{g}\in\alg{H}^{\Fr}\\ \bar{g}^{-1}s\bar{g}\in \alg{T}_{\alg{H}}^{\Fr}}}\theta_\alg{H}(\bar{g}^{-1}s\bar{g})Q^{C_\alg{G}(s)^\circ}_{\bar{g}\alg{T}\bar{g}^{-1}}(u).
\end{align*}

We have that $m=[\alg{T}^{\Fr}:\alg{T}_{\alg{H}}^{\Fr}]$, but this is equal to ${[(C_{\alg{G}}(s)^\circ)^{\Fr}:(C_\alg{H}(s)^\circ)^{\Fr}]}$. Indeed, because $s\in ({^g\alg{T}})^{\Fr}$, $({^g\alg{T}})^{\Fr} \subset (C_\alg{G}(s)^\circ)^{\Fr}$ so that ${\quo{\alg{T}^{\Fr}}{\alg{T}_\alg{H}^{\Fr}} \simeq \quo{\alg{G}^{\Fr}}{\alg{H}^{\Fr}} = \quo{(C_\alg{G}(s)^\circ)^{\Fr}\alg{H}^{\Fr}}{\alg{H}^{\Fr}} \simeq \quo{(C_{\alg{G}}(s)^\circ)^{\Fr}}{(C_{\alg{H}}(s)^\circ)^{\Fr}}}$. It follows that ${\frac{m}{|(C_{\alg{G}}(s)^\circ)^{\Fr}|} = \frac{1}{|(C_{\alg{H}}(s)^\circ)^{\Fr}|}}$. Therefore,
$$R_{\alg{T,}\theta}|_{\alg{H}^{\Fr}}(g) = \frac{1}{|(C_\alg{H}(s)^\circ)^{\Fr}|}\sum\limits_{\substack{\bar{g}\in\alg{H}^{\Fr}\\ \bar{g}^{-1}s\bar{g}\in \alg{T}_{\alg{H}}^{\Fr}}}\theta_\alg{H}(\bar{g}^{-1}s\bar{g})Q^{C_\alg{G}(s)^\circ}_{\bar{g}\alg{T}\bar{g}^{-1}}(u).$$ It remains to show that $Q_{\bar{g}\alg{T}\bar{g}^{-1}}^{C_\alg{G}(s)^\circ}(u) = Q_{\bar{g}\alg{T}_{\alg{H}}\bar{g}^{-1}}^{C_{\alg{H}}(s)^\circ}(u)$ for all $\bar{g}\in\alg{H}^{\Fr}$ to complete the proof.

Let $\ad: C_\alg{G}(s)^\circ \rightarrow \quo{C_\alg{G}(s)^\circ}{Z(C_\alg{G}(s)^\circ)}$ be the projection map of $C_\alg{G}(s)^\circ$ into the adjoint group $\ad(C_\alg{G}(s)^\circ) =\quo{C_\alg{G}(s)^\circ}{Z(C_\alg{G}(s)^\circ)}$. Similarly, we let $\ad_{\alg{H}}$ be the projection map of $C_\alg{H}(s)^\circ$ into the adjoint group $\ad_{\alg{H}}(C_\alg{H}(s)^\circ)=\quo{C_\alg{H}(s)^\circ}{Z(C_\alg{H}(s)^\circ)}$. According to \cite[Formula 4.1.1]{DL:1976}, we have that $Q_{\bar{g}\alg{T}\bar{g}^{-1}}^{C_\alg{G}(s)^\circ}(u) = Q_{\ad(\bar{g}\alg{T}\bar{g}^{-1})}^{{\ad(C_\alg{G}(s)^\circ)}}(\ad(u))$ and $Q_{\bar{g}\alg{T}_{\alg{H}}\bar{g}^{-1}}^{C_{\alg{H}}(s)^\circ}(u) = Q_{\ad_{\alg{H}}(\bar{g}\alg{T}_{\alg{H}}\bar{g}^{-1})}^{\ad_{\alg{H}}(C_{\alg{H}}(s)^\circ)}(\ad_{\alg{H}}(u))$. One can verify that the map that sends  $cZ(C_{\alg{H}}(s)^\circ)$ to $cZ(C_{\alg{G}}(s)^\circ)$ for all $c\in C_{\alg{H}}(s)^\circ$ defines an isomorphism $\ad_{\alg{H}}(C_\alg{H}(s)^\circ) \simeq \ad(C_\alg{G}(s)^\circ)$. Furthermore, this isomorphism maps $\ad_{\alg{H}}(\bar{g}\alg{T}_{\alg{H}}\bar{g}^{-1})$ to $\quo{\bar{g}\alg{T}_\alg{H}\bar{g}^{-1}Z(C_\alg{G}(s)^\circ)}{Z(C_\alg{G}(s)^\circ)}$. Given that $\alg{G}=Z(\alg{G})^\circ\alg{H}$, it follows that $\alg{T} = Z(\alg{G})^\circ\alg{T}_\alg{H}$ so that $\quo{\bar{g}\alg{T}_\alg{H}\bar{g}^{-1}Z(C_\alg{G}(s)^\circ)}{Z(C_\alg{G}(s)^\circ)} = \ad(\bar{g}\alg{T}\bar{g}^{-1})$. It then follows from Theorem \ref{th:iso} that $R_{\ad_{\alg{H}}(\bar{g}\alg{T}_{\alg{H}}\bar{g}^{-1}),1}^{\ad_{\alg{H}}(C_{\alg{H}}(s)^\circ)}(\ad_{\alg{H}}(u)) = R_{\ad(\bar{g}\alg{T}\bar{g}^{-1}),1}^{\ad(C_{\alg{G}}(s)^\circ)}(\ad(u))$, that is, ${Q_{\ad_{\alg{H}}(\bar{g}\alg{T}_{\alg{H}}\bar{g}^{-1})}^{\ad_{\alg{H}}(C_{\alg{H}}(s)^\circ)}(\ad_{\alg{H}}(u)) = Q_{\ad(\bar{g}\alg{T}\bar{g}^{-1})}^{\ad(C_{\alg{G}}(s)^\circ)}(\ad(u))}$, and therefore $Q_{\bar{g}\alg{T}\bar{g}^{-1}}^{C_\alg{G}(s)^\circ}(u) = Q_{\bar{g}\alg{T}_{\alg{H}}\bar{g}^{-1}}^{C_{\alg{H}}(s)^\circ}(u)$.
\end{proof}

\section*{Acknowledgments}
\addcontentsline{toc}{section}{Acknowledgments}
I would like to thank my thesis supervisor, Dr. Monica Nevins, for her guidance and support throughout the entirety of my research project, and for providing feedback on this paper. I would also like to thank Dr. Peter Latham for the many discussions on types, and Dr. Jeffrey Adler for the email correspondence regarding his multiplicity results. I would also like to thank the participants of the session "Representation Theory of Groups Defined Over Local Fields" during the 2019 Canadian Mathematical Society Summer Meeting in Regina for the feedback on my research presentation and the various fruitful discussions.


\textsc{Department of Mathematics and Statistics, University of Ottawa,
Ottawa, ON, Canada K1N 6N5}

\textit{E-mail address:} \href{mailto:abour115@uottawa.ca}{abour115@uottawa.ca}

\end{document}